\documentclass[11pt]{amsart}

\usepackage[dvipsnames]{xcolor}
\usepackage{tikz}
\usetikzlibrary{positioning}
\usepackage{amsmath}
\usepackage{amssymb}
\usepackage[T1]{fontenc}
\usepackage[utf8]{inputenc}
\usepackage{mathtools}
\usepackage{amsfonts}
\usepackage{fancyhdr}
\usepackage{mathrsfs}
\usepackage{verbatim}
\usetikzlibrary{arrows}

\begin{document}

\title{A Tannakian Reconstruction Theorem for IndBanach Spaces}
\author{Kobi Kremnizer \& Craig Smith}

\maketitle

\begin{abstract}
Classically, Tannaka-Krein duality allows us to reconstruct a (co)algebra from its category of representation. In this paper we present an approach that allows us to generalise this theory to the setting of Banach spaces. This leads to several interesting applications in the directions of analytic quantum groups, bounded cohomology and galois cohomology. A large portion of this paper is dedicated to such examples.
\end{abstract}

\tableofcontents

\newtheorem{theorem}{Theorem}[section]
\newtheorem{corollary}[theorem]{Corollary}
\newtheorem{example}[theorem]{Example}
\newtheorem{lem}[theorem]{Lemma}
\newtheorem{obs}[theorem]{Observation}
\newtheorem{ass}[theorem]{Assumptions}
\newtheorem{prop}[theorem]{Proposition}
\theoremstyle{definition}
\newtheorem{defn}[theorem]{Definition}

\newtheorem{rem}[theorem]{Remark}
\numberwithin{equation}{section}

\newenvironment{definition}[1][Definition]{\begin{trivlist}
\item[\hskip \labelsep {\bfseries #1}]}{\end{trivlist}}

\newenvironment{remark}[1][Remark]{\begin{trivlist}
\item[\hskip \labelsep {\bfseries #1}]}{\end{trivlist}}

\newenvironment{notation}[1][Notation]{\begin{trivlist}
\item[\hskip \labelsep {\bfseries #1}]}{\end{trivlist}}

\setcounter{section}{-1}

\section{Introduction}
Classically, Tannaka-Krein duality answers the questions of whether a compact topological group (or affine group scheme as in \cite{TC}, \cite{CT}) can be recovered from its category of linear representations, and of when a category (with an appropriate fibre functor) is equivalent to representations of such a group.  The answer to these questions can be seen as an application of the Barr-Beck theorem, along with the fact that a cocontinuous linear functor on the category of vector spaces must be of the form $V \otimes -$ for some space $V$. This second point follows from the fact that any vector space is a colimit of copies of the base field.\\

Unfortunately, the above is not true for the category of Banach spaces. However, the contracting category of Banach spaces does have an analogous property, and so a brief investigation of contracting colimits in Section 1 allows us to proceed as before. We also note that the category of Banach spaces is neither complete nor cocomplete, and so we instead work in its Ind completion.  Using this, we deduce an analogue of Tannaka duality for IndBanach spaces in Section 2.\\

In Section 3 we demonstrate some examples of applications of this theory. These include a short exploration of different analytic gradings, which the authors hope will be their first steps towards defining analytic quantum groups, and conclude with the example of Galois descent for categories of IndBanach spaces. Perhaps the most fruitful example, however, involves representations of topological groups. In \cite{AFoBC}, Bühler shows that continuous bounded cohomology of a group $G$ comes from the derived invariants functor on a quasi-ableian category which we denote $G\text{-Mod}^{\text{iso}}$. In Section \ref{TopologicalGroupsExample} we show that this is a category of coalgebras over a comonadic functor (or comodules of an IndBanach bialgebra when the group is compact).  We may therefore rephrase bounded cohomology in terms of cohomology of a monoidal comonadic functor (or an IndBanach bialgebra).

\subsection*{Funding}
This work was supported by the Engineering and Physical Sciences Research Council [EP/M024830/1 to KK, EP/M50659X/1 to CS].

\section{Preliminaries and notation}
We begin with some preliminaries on category theory. For more details see Borceaux's \emph{Handbook of Categorical Algebra 2} \cite[p.~189-197]{HCA2}.

\begin{defn}
A \emph{monad} on a category $\mathcal{C}$ is a triple $\mathbb{T} = (T, \eta, \mu)$ where $T: \mathcal{C} \rightarrow \mathcal{C}$ is a functor and $\eta : \text{id}_{\mathcal{C}} \Rightarrow T$, $\mu : T \circ T \Rightarrow T$ are natural transformations satisfying the usual associativity and unit constraints as for an algebra. An \emph{algebra} on this monad is a pair $(C, \xi)$ where $C$ is an object in the category and $\xi : T(C) \rightarrow C$ is a morphism in the category satisfying appropriate compatibility requirements. A \emph{morphism of algebras} $f:(C, \xi) \rightarrow (C', \xi')$ is a morphism $f:C \rightarrow C'$ in the category such that $f \circ \xi = \xi' \circ T(f)$. These algebras in $\mathcal{C}$ over a monad $\mathbb{T}$ form a category, denoted $\mathcal{C}^{\mathbb{T}}$, known as the \emph{Eilenberg-Moore category} of the monad. Dually, we define \emph{comonads} $\mathbb{U}$ and their analogous Eilenberg-Moore categories of \emph{coalgebras} $\mathcal{C}_{\mathbb{U}}$.
\end{defn}

\begin{prop}[\cite{HCA2}]
Suppose we have a pair of adjoint functors
$$F: \mathcal{C} \longleftrightarrow \mathcal{D}:G.$$
with unit $\eta : \text{id}_{\mathcal{C}} \Rightarrow G\circ F$ and counit $\varepsilon : F \circ G \Rightarrow \text{id}_{\mathcal{D}}$. Then $\mathbb{T} = (T := G\circ F, \eta, \mu)$ defines a monad where $\mu$ is the horisontal composition $\mu = \text{id}_{G} \ast \varepsilon \ast \text{id}_{F}:GFGF \Rightarrow  G \circ \text{id}_{\mathcal{D}} \circ F = GF $. Similarly, $\mathbb{U} = (U := F\circ G, \varepsilon, \Delta)$ forms a comonad where $\Delta := \text{id}_{F} \ast \eta \ast \text{id}_{G}$. Furthermore, we have \emph{comparison functors} $K^{\mathbb{T}} : \mathcal{D} \rightarrow \mathcal{C}^{\mathbb{T}}$, $J_{\mathbb{U}} : \mathcal{C} \rightarrow \mathcal{D}_{\mathbb{U}}$ defined respectively by
$$\begin{array}{rl}
K^{\mathbb{T}}(A)  =  (G(A),G(\varepsilon_{A})), & K^{\mathbb{T}}(f)  =  G(f),  \\
J_{\mathbb{U}}(B)  =  (F(B),F(\eta_{B})), & J_{\mathbb{U}}(g)  =  F(g),  \\ \end{array}$$
for all objects $A$ in $\mathcal{D}$ and $B$ in $\mathcal{C}$ and for all morphisms $f$ in $\mathcal{D}$ and $g$ in $\mathcal{C}$.
\end{prop}

\begin{defn}
A functor $G: \mathcal{D} \rightarrow \mathcal{C}$ is called \emph{monadic}  if there exists a monad $\mathbb{T} = (T, \eta, \mu)$ on $\mathcal{C}$ and an equivalence of categories $J: \mathcal{D} \rightarrow \mathcal{C}^{\mathbb{T}}$ such that $F \circ J$ is isomorphic as a functor to $G$, where $F: \mathcal{C}^{\mathbb{T}} \rightarrow \mathcal{C}$ is the forgetful functor. Equivalently, $G$ is monadic if it has a left adjoint $F: \mathcal{C} \rightarrow \mathcal{D}$, and so the pair form a monad $\mathbb{T} = (T := G\circ F, \eta, \mu)$ on $\mathcal{C}$, and if the comparison functor $K^{\mathbb{T}} : \mathcal{D} \rightarrow \mathcal{C}^{\mathbb{T}}$ is an equivalence of categories. Dually, a functor $F: \mathcal{C} \rightarrow \mathcal{D}$ is \emph{comonadic} if it has a right adjoint $G: \mathcal{D} \rightarrow \mathcal{C}$, and so form a comonad $\mathbb{U} = (U := F\circ G, \varepsilon, \Delta)$ on $\mathcal{D}$, and if the comparison functor $J_{\mathbb{U}} : \mathcal{C} \rightarrow \mathcal{D}_{\mathbb{U}}$ is an equivalence of categories.
\end{defn}

The following result, sometimes known as \emph{Beck's Monadicity Theorem}, gives criterion for when a functor is monadic (or comonadic). 

\begin{theorem}
\emph{(The Barr-Beck Theorem \cite[p.~212]{HCA2})}
\label{BarrBeck}
A functor $G: \mathcal{D} \rightarrow \mathcal{C}$ is monadic if and only if
\begin{itemize}
\item[i)]$G$ has a left adjoint $F$;
\item[ii)]$G$ reflects isomorphisms. That is, if $G(f)$ is an isomorphism then $f$ is an isomorphism for all morphisms $f$; and
\item[iii)]given a pair $f,g : A \rightarrow B$ of morphisms in $\mathcal{D}$ such that $G(f), G(g)$ have a split coequaliser $d: G(B) \rightarrow D$ in $\mathcal{C}$ then $f,g$ have a coequaliser $c: B \rightarrow C$ in $\mathcal{D}$ such that $G(c)=d, G(C) = D$.
\end{itemize}
\end{theorem}

A dual version of the Barr-Beck theorem then characterises comonadic functors as follows.

\begin{theorem}
\label{DualBarrBeck}
A functor $F: \mathcal{C} \rightarrow \mathcal{D}$ is comonadic if and only if
\begin{itemize}
\item[i)]$F$ has a right adjoint $G$;
\item[ii)]$F$ reflects isomorphisms; and
\item[iii)]given a pair $f,g : A \rightarrow B$ are morphisms in $\mathcal{C}$ such that $F(f), F(g)$ have a split equaliser $h: H \rightarrow F(A)$ in $\mathcal{D}$ then $f,g$ have an equaliser $e: E \rightarrow A$ in $\mathcal{C}$ such that $F(e)=h, F(E) = H$.
\end{itemize}
\end{theorem}

\section{Contracting (co)products}
Fix a complete valued field $k$ with non-trivial valuation, either Archimedean or non-Archimedean.

\begin{defn}
\label{BanachCategory}
Let $\text{Ban}_{k}$ denote the category of $k$-Banach spaces, each equipped with a specific norm, and bounded linear transformations between them. Let $\text{Ban}_{k}^{\leq 1}$ denote the wide subcategory whose morphisms are bounded linear transformations of norm at most 1. By \emph{wide} we mean that $\text{Ban}_{k}^{\leq 1}$ contains all objects of $\text{Ban}_{k}$. If our field is non-Archimedean then Banach spaces may be defined in two ways, depending on whether we require norms to satisfy the usual triangle inequality or the strong triangle inequality. For most of this paper we will be able to treat both of these definitions uniformly, and will refer to them as the Archimedean and non-Archimedean cases respectively when they differ.
\end{defn}

\begin{defn}
\label{Contracting(Co)Products}
Let $(V_{i})_{i \in I}$ be a family of Banach spaces. Let us define the \emph{contracting product} of this family as the Banach space
$$\prod\nolimits_{i \in I}^{\leq 1}V_{i}=\{(v_{i})_{i \in I} \in \times_{i \in I} V_{i} \mid \text{Sup}_{i \in I} \|v_{i}\| \leq \infty\}$$
with norm $\|(v_{i})\|=\text{Sup}_{i \in I} \|v_{i}\|$ in both the Archimedean and non-Archimedean cases, and the \emph{contracting coproduct} as the Banach space
$$\coprod\nolimits_{i \in I}^{\leq 1}V_{i} = \{(v_{i})_{i \in I} \in \times_{i \in I} V_{i} \mid \sum_{i \in I} \|v_{i}\| \leq \infty\}$$
with norm $\|(v_{i})\|=\sum_{i \in I} \|v_{i}\|$ in the Archimedean case and
$$\coprod\nolimits_{i \in I}^{\leq 1}V_{i} = \{(v_{i})_{i \in I} \in \times_{i \in I} V_{i} \mid \text{lim}_{i \in I} \|v_{i}\| =0\}$$
with norm $\|(v_{i})\|=\text{Sup}_{i \in I} \|v_{i}\|$ in the non-Archimedean case.
\end{defn}

\begin{prop}
The category $\text{Ban}_{k}^{\leq 1}$ has small limits and colimits.
\end{prop}

\begin{proof}
Indeed, it has kernels and cokernels inhereted from $\text{Ban}_{k}$, and it is straightforward to check that Definition \ref{Contracting(Co)Products} describes products and coproducts in this category.
\end{proof}

\begin{defn}
These limits and colimits give objects in $\text{Ban}_{k}$. We shall refer to them as \emph{contracting limits and colimits} respectively, and denotate them by $\text{lim}_{I}^{\leq 1}$ and $\text{colim}_{I}^{\leq 1}$.
\end{defn}

\begin{remark}
Note that filtered contracting colimits are not left exact. For example, the maps $k \rightarrow k_{\frac{1}{n}}$ are all isomorphisms in $\text{Ban}_{k}$ (and bimorphisms in $\text{Ban}_{k}^{\leq 1}$) but taking contracting colimits over $n \geq 1$ we obtain the morphism $k \rightarrow \{0\}$.
\end{remark}

Contracting (co)products have the following universal property in $\text{Ban}_{k}$.

\begin{lem}
\label{ContractingUniversalProperty}
For all collections of morphisms $\{f_{i}:U \rightarrow V_{i}\}_{i \in I}$ (respectively $\{g_{i}:V_{i} \rightarrow W\}_{i \in I}$) such that $\{ \|f_{i}\|\}_{i \in I}$ is bounded (respectively $\{ \|g_{i}\|\}_{i \in I}$ is bounded) by some $M>0$, there exists a unique map $U \rightarrow \prod_{i \in I}^{\leq 1}V_{i}$ (respectively $\coprod_{i \in I}^{\leq 1}V_{i} \rightarrow W$) of norm as most $M$ such that $f_{i}$ is the composite $U \rightarrow \prod_{j \in I}^{\leq 1}V_{j} \rightarrow V_{i}$ (respectively $g_{i}$ is the composite $V_{i} \rightarrow \prod_{j \in I}^{\leq 1}V_{j} \rightarrow U$). That is, $\underline{\text{Hom}}(U,\prod_{i \in I}^{\leq 1}V_{i}) \cong \prod_{i \in I}^{\leq 1}\underline{\text{Hom}}(U,V_{i})$ and $\underline{\text{Hom}}(\coprod_{i \in I}^{\leq 1}V_{i},W) \cong \prod_{i \in I}^{\leq 1}\underline{\text{Hom}}(V_{i},W)$.
\end{lem}

\begin{proof}
As the valuation on our field is assumed to be non-trivial, we may take $M \in |k^{\times}|$ without loss of generality, so there is $\lambda \in k^{\times}$ with $|\lambda|=M$. Then we may rescale our family of morphisms to $\{\frac{f_{i}}{\lambda}\}_{i \in I}$ in $\text{Ban}_{k}^{\leq 1}$. By the universal property we get a map $\phi:U \rightarrow \prod_{i \in I}^{\leq 1}V_{i}$ of modulus at most $1$, and scaling by $\lambda$ gives our desired map, $\lambda \cdot \phi$. The proof for contracting coproducts is similar.
\end{proof}

\begin{defn}
For a set $I$, let $\text{Ban}_{k}^{I,\text{bd}}$ be the category whose objects are collections $(V_{i})_{i \in I}$ of Banach spaces $V_{i}$ indexed by $i \in I$ and whose morphisms are uniformly bounded,
$$\text{Hom}((V_{i})_{i \in I},(V'_{i})_{i \in I}):=\prod\nolimits_{i \in I}^{\leq 1}\underline{\text{Hom}}(V_{i},V'_{i}).$$
\end{defn}

It follows from Lemma \ref{ContractingUniversalProperty} that $\prod_{i \in I}^{\leq 1}$ and $\coprod_{i \in I}^{\leq 1}$ define functors from $\text{Ban}_{k}^{I,\text{bd}}$ to $\text{Ban}_{k}$. Furthermore, contracting products are right adjoints to the diagonal functors
$$\Delta^{I}:\text{Ban}_{k} \rightarrow \text{Ban}_{k}^{I,\text{bd}}, \quad V \mapsto (V)_{i \in I},$$
and likewise contracting coproducts are left adjoints to $\Delta^{I}$.

\begin{remark}
Note that contracting products and contracting coproducts do not necessarily commute. For example the natural map
$$\coprod_{i \in \mathbb{Z}}^{\leq 1} \prod_{j \in \mathbb{Z}}^{\leq 1} k \rightarrow \prod_{j \in \mathbb{Z}}^{\leq 1} \coprod_{i \in \mathbb{Z}}^{\leq 1} k$$
is not surjective, as $(\delta_{i,j})_{i,j \in \mathbb{Z}}$ is not in the image.
\end{remark}

\begin{defn}
Let $\text{IndBan}_{k}$ be the Ind completion of $\text{Ban}_{k}$. That is, $\text{IndBan}_{k}$ is the category whose objects are filtered diagrams $X:I \rightarrow \text{Ban}_{k}$ of Banach spaces, with morphisms
$$\text{Hom}(X,Y)= \text{lim}_{i \in I}\text{colim}_{j \in J} \text{Hom}(X(i),Y(j)).$$
We think of these diagrams as formal colimits, and hence use the notation $\text{"colim"}_{i \in I} X(i)$ for the diagram $X$. For a Banach space $V$ we will often denote by $\text{"}V\text{"}$ the object in $\text{IndBan}_{k}$ represented by the constant singleton diagram at $V$, and often just as $V$ when there is no ambiguity.
\end{defn}

\begin{defn}
We will say that a category $\mathcal{C}$ is \emph{locally presentable} if it is cocomplete and has a small full subcategory $\mathcal{C}_{0}$ of compact objects such that every object in $\mathcal{C}$ is canonically a colimit of objects in $\mathcal{C}_{0}$.
\end{defn}

\begin{prop}
The category $\text{IndBan}_{k}$ is a complete and cocomplete, locally presentable, quasi-abelian category, and can be given a closed monoidal structure extending that of $\text{Ban}_{k}$ by defining
$$(\text{"colim"}_{i \in I} X_{i})\hat{\otimes}(\text{"colim"}_{j \in J} Y_{j}):= \text{"colim"}_{\substack{i \in I \\ j \in J}} X_{i} \hat{\otimes} Y_{j},$$
$$\underline{\text{Hom}}(\text{"colim"}_{i \in I} X_{i},\text{"colim"}_{j \in J} Y_{j}):=  \text{lim}_{i \in I}\text{colim}_{j \in J} \underline{Hom}(X_{i},Y_{j}).$$
\end{prop}

\begin{proof}
Explicit construction of limits can be found in Section 1.4.1 of \cite{LaACH}. By construction, $\text{IndBan}_{k}$ is locally presentable with compact objects $\text{Ban}_{k}$.
\end{proof}

\begin{remark}
For an account of Ind completions see \cite{CaS}, and more on $\text{IndBan}_{k}$ can be found in \cite{TRTfDM}, \cite{SDiBAG}, \cite{NAAGaRAG} and \cite{LaACH} and numerous other excellent sources. A thorough exposition of quasi-abelian categories can be found in \cite{QACS}. Results about locally presentable categories, including the Adjoint Functor Theorem (from which Theorem \ref{AdjointFunctorTheorem} in the following is adapted), can be found in \cite{LPaAC}.
\end{remark}

\begin{defn}
We extend the definition of contracting (co)products to $\text{IndBan}_{k}$ as follows. The contracting product and coproduct functors
$$\prod\nolimits^{\leq 1}_{I},\coprod\nolimits^{\leq 1}_{I}:\text{Ban}_{k}^{I,\text{bd}} \rightarrow \text{Ban}_{k}$$
induce functors from the Ind completion of $\text{Ban}_{k}^{I,\text{bd}}$,
$$\text{IndBan}_{k}^{I,\text{bd}}:=\text{Ind}(\text{Ban}_{k}^{I,\text{bd}}),$$
to $\text{IndBan}_{k}$, which we will continue to denote as $\prod^{\leq 1}_{I}$ and $\coprod^{\leq 1}_{I}$ respectively. There is a faithful diagonal embedding functor $\Delta^{I}:\text{IndBan}_{k} \rightarrow \text{IndBan}_{k}^{I,\text{bd}}$ induced by $\Delta^{I}:\text{Ban}_{k} \rightarrow \text{Ban}_{k}^{I,\text{bd}}$.
\end{defn}

\begin{remark}
The embedding $\text{Ban}_{k}^{I, \text{bd}} \hookrightarrow \text{Ban}_{k}^{I}$ induces a faithful embedding $\text{IndBan}_{k}^{I,\text{bd}} \rightarrow \text{Ind}(\text{Ban}_{k}^{I}) \cong \text{IndBan}_{k}^{I}$. This allows us to think of objects of $\text{IndBan}_{k}^{I,\text{bd}}$ as collections of IndBanach spaces indexed over $I$ which can be expressed as formal colimits in a uniformly bounded way, and morphisms being uniformly bounded.
\end{remark}

\begin{prop}
With the above definitions, there are adjunctions
$$\text{Hom}(\coprod\nolimits_{I} ^{\leq 1}X_{I},Y) \cong \prod\nolimits_{i \in I} ^{\leq 1}\text{Hom}(X_{I},\Delta^{I}Y)$$
and
$$\text{Hom}(Y,\prod\nolimits_{I} ^{\leq 1}X_{I}) \cong \prod\nolimits_{i \in I} ^{\leq 1}\text{Hom}(\Delta^{I}Y,X_{I})$$
for $X_{I} \in \text{IndBan}_{k}^{I,\text{bd}}$, $Y \in \text{IndBan}_{k}$.
\end{prop}
\begin{proof}
This follows from the adjunction given in Lemma \ref{ContractingUniversalProperty} by taking filtered colimits.
\end{proof}

\begin{defn}
We will say that a functor $\mathcal{F}:\text{IndBan}_{k} \rightarrow \text{IndBan}_{k}$ \emph{commutes with contracting coproducts} if the functors $\mathcal{F}^{I}:\text{IndBan}_{k}^{I} \rightarrow \text{IndBan}_{k}^{I}$ restrict to functors $\mathcal{F}^{I}:\text{IndBan}_{k}^{I,\text{bd}} \rightarrow \text{IndBan}_{k}^{I,\text{bd}}$ under the embedding $\text{IndBan}_{k}^{I,\text{bd}} \hookrightarrow \text{IndBan}_{k}^{I}$ such that the diagram of functors
\begin{center}
\begin{tikzpicture}[node distance=6cm, auto]
  \node (A) {$\text{IndBan}_{k}^{I,\text{bd}}$};
  \node (B) [below=1cm of A] {$\text{IndBan}_{k}$};
  \node (C) [right=1.3cm of A] {$\text{IndBan}_{k}^{I,\text{bd}}$};
  \node (D) [below=1cm of C] {$\text{IndBan}_{k}$};
  \draw[->] (A) to node [swap]{$\coprod\nolimits^{\leq 1}_{I}$} (B);
  \draw[->] (A) to node {$\mathcal{F}^{I}$} (C);
  \draw[->] (B) to node {$\mathcal{F}$} (D);
  \draw[->] (C) to node {$\coprod\nolimits^{\leq 1}_{I}$} (D);
  \draw[double,double equal sign distance,-implies,shorten >=10pt,shorten <=10pt] 
  (C) to node {} (B);
\end{tikzpicture}
\end{center}
commutes up to a natural isomorphism.
\end{defn}

\begin{remark}
It is important to note that, since contracting coproducts are not functorial on $\text{IndBan}_{k}^{I}$, or even on the full subcategory on the essential image of $\text{IndBan}_{k}^{I,\text{bd}}$, the statement of whether or not a functor commutes with contracting coproducts is not invariant under isomorphism. However, the following weaker notion is invariant under isomorphism of functors.
\end{remark}

\begin{defn}
\label{CommutingWithl1}
For a set $S$ we will denote by $l^{1}(S)$ the contracting coproduct $l^{1}(S):= \coprod^{\leq 1}_{S}k$. We will say that a functor  \emph{commutes with $l^{1}$} if the natural map
$$\coprod\nolimits^{\leq 1}_{S} F(k) \xrightarrow{\sim} F(l^{1}(S))$$
is an isomorphism. This map is the image of the identity under the composition
$$\begin{array}{rcl}
\text{Hom}(\coprod_{S}^{\leq 1}k, \coprod_{S}^{\leq 1}k)&\cong& \prod_{S}^{\leq 1}\text{Hom}(k, \coprod_{S}^{\leq 1}k)\\
&\rightarrow& \prod_{S}^{\leq 1}\text{Hom}(F(k), F(\coprod_{S}^{\leq 1}k))\\
&\cong& \text{Hom}(\coprod_{S}^{\leq 1}F(k), F(\coprod_{S}^{\leq 1}k)).
\end{array}$$
\end{defn}

\section{Categories of IndBanach (co)modules}

\subsection{IndBanach modules of IndBanach algebras}

\begin{defn}
Let $\mathcal{C}$ be a locally presentable, quasi-abelian category enriched over $\text{IndBan}_{k}$ and let $F:\mathcal{C} \rightarrow \text{IndBan}_{k}$ be an enriched functor. We say that $F$ is a \emph{fibre functor} over $\text{IndBan}_{k}$ if $F$ is bicontinuous, strongly exact, faithful and reflects strict morphisms.
\end{defn}

The following adaptation of the Adjoint Functor Theorem for locally presentable categories (see \cite{LPaAC}) tells us when an enriched adjoint functor exists.

\begin{theorem}[Enriched Adjoint Functor Theorem \cite{BCoECT}]
\label{AdjointFunctorTheorem}
Let $\mathscr{F}:\mathscr{C} \rightarrow \mathscr{D}$ be a functor between locally presentable categories, enriched over $\text{IndBan}_{k}$. Then $\mathscr{F}$ has an enriched right adjoint if and only if it preserves all small colimits. If $\mathscr{C}$ is complete and $\mathscr{F}$ also preserves all small limits then $\mathscr{F}$ has an enriched left adjoint.
\end{theorem}

\begin{proof}
This follows directly from Theorem 5.32 and Theorem 5.33 in \cite{BCoECT}.
\end{proof}

This gives us the following Lemma.

\begin{lem}
\label{FibreFunctorsMonadic}
Let $\mathcal{C}$ be a locally presentable, quasi-abelian category, and let $F:\mathcal{C} \rightarrow \text{IndBan}_{k}$ be a fibre functor over $\text{IndBan}_{k}$. Then $F$ satisfies the conditions of Barr-Beck (Theorem \ref{BarrBeck}), so $\mathcal{C}$ is equivalent to the category of algebras of a monadic functor $T$ on $\text{IndBan}_{k}$.
\end{lem}
\begin{proof}
By Theorem \ref{AdjointFunctorTheorem}, since $\text{IndBan}_{k}$ is locally presentable and a fibre functor $F$ is both continuous and cocontinuous it has a left adjoint, $G$. Hence property (i) of Theorem \ref{BarrBeck} is satisfied. For property (ii), if $f:A \rightarrow B$ is a morphism in $\mathcal{C}$ such that $Ff$ is an isomorphism then it fits into a strictly coexact sequence $A \overset{f}{\rightarrow} B \rightarrow \text{Coker}(f)$, the image of which under $F$ is then also strictly coexact, so $F(\text{Coker}(f))=0$. A similar argument shows $F(\text{Ker}(f))=0$. Since $F$ is faithful, this means that $f$ has trivial kernel and cokernel. It then follows from the fact that $F$ reflects strictness that $f$ is also an isomorphism. $\mathcal{C}$ is quasi-abelian and hence has equalisers, and so (iii) follows from the strong exactness of $F$. Thus, by Theorem \ref{BarrBeck}, $F$ is monadic and hence $\mathcal{C}$ is equivalent to the category of algebras of $T=FG$.
\end{proof}

\begin{lem}
\label{FibreFunctorMonadCocontinuous}
With conditions as in the previous lemma, the monad $T$ is cocontinuous.
\end{lem}

\begin{proof}
This follows from the fact that $F$ is assumed to be cocontinuous and $G$ is a left adjoint, hence also cocontinuous.
\end{proof}

\begin{lem}
\label{TensorFunctorClassification}
A functor $\mathscr{V}:\text{IndBan}_{k} \rightarrow \text{IndBan}_{k}$ is naturally isomorphic to one of the form $V \hat{\otimes} -$ for an IndBanach space $V$ if and only if $\mathscr{V}$ is enriched over $\text{IndBan}_{k}$, cocontinuous and commutes with $l^{1}$.
\end{lem}

\begin{proof}
For a Banach space $V$, $V \hat{\otimes} -$ is a left adjoint on both $\text{Ban}_{k}$ and $\text{Ban}_{k}^{\leq 1}$ hence is cocontinuous and commutes with contracting coproducts. Since contracting coproducts commute with colimits, this is also true for any IndBanach space $V$. Hence $V\hat{\otimes}-$ commutes with $l^{1}$.\\

Conversely, suppose $\mathscr{V}:\text{IndBan}_{k} \rightarrow \text{IndBan}_{k}$ is enriched, cocontinuous and commutes with $l^{1}$. Let $W$ be a Banach space which, by Lemma A.39 of \cite{NAAGaRAG}, can be written as the cokernel of a morphism
$$f:P(W') \rightarrow P(W)$$
where
$$P(X):= \coprod\nolimits^{\leq 1}_{\substack{x \in X \\ \|x\|=1}}k=l^{1}(\{x \in X \mid \|x\|=1\})$$
for any Banach space $X$, and $W'$ is the kernel of the natural map $I(W) \twoheadrightarrow W$. But, since $\mathscr{V}$ commutes with $l^{1}$,
$$\mathscr{V}(P(X)) \cong \coprod\nolimits^{\leq 1}_{\substack{x \in X \\ \|x\|=1}}\mathscr{V}(k) \cong \mathscr{V}(k) \hat{\otimes} P(X)$$
for all sets $X$. The map $f$ is induced by uniformly bounded maps ${f_{x}:k \rightarrow P(W)}$ indexed over $x \in W'$ with $\|x\|=1$. Each $f_{x}$ is a convergent sum $\sum_{y}a_{x,y}\iota_{y}$, $a_{x,y} \in k$, indexed over $y \in W$ with $\|y\|=1$, where $\iota_{y}$ injects the copy of $k$ indexed by $y$ into $P(W)$. Since $\mathscr{V}$ is enriched, if $\mathscr{V}(k)=\text{"colim"}_{i \in I}X_{i}$ and $\mathscr{V}(P(W))=\text{"colim"}_{j \in J} Y_{j}$, then the map
$$\text{Hom}(k,P(W)) \xrightarrow{\mathscr{V}} \text{Hom}(\mathscr{V}(k),\mathscr{V}(P(W)))$$
is given by a compatible collection of continuous maps of Banach spaces
$$\text{Hom}(k,P(W)) \xrightarrow{\mathscr{V}_{i}} \text{Hom}(X_{i},Y_{j_{i}})$$
for each $i \in I$ and for some corresponding $j_{i} \in J$. Then
$$\mathscr{V}_{i}(f_{x}) = \mathscr{V}_{i}(\sum_{y}a_{x,y}\iota_{y})=\sum_{y}a_{x,y}\mathscr{V}_{i}(\iota_{y})$$
as maps $X_{i} \rightarrow Y_{j_{i}}$ for each $i \in I$. By construction of the morphism in Definition \ref{CommutingWithl1}, the map $\mathscr{V}(\iota_{y})$ is equal to the composition
$$\mathscr{V}(k) \cong \mathscr{V}(k) \hat{\otimes} k \xrightarrow{\text{Id} \otimes \iota_{y}} \mathscr{V}(k) \hat{\otimes} P(W) \cong \mathscr{V}(P(W)).$$
By potentially replacing each $j_{i}$ with a larger element in the filtered set $J$, we may assume that the isomorphism $\mathscr{V}(k) \hat{\otimes} P(W) \xrightarrow{\sim} \mathscr{V}(P(W))$ is given by a collection of maps $X_{i} \hat{\otimes} P(W) \rightarrow Y_{j_{i}},$ where the composition
$$X_{i} \cong X_{i} \hat{\otimes} k \xrightarrow{\text{Id} \otimes \iota_{y}} X_{i} \hat{\otimes} P(W) \rightarrow Y_{j_{i}}$$
 is equal to $\mathscr{V}_{i}(\iota_{y})$. Thus the diagram
\begin{center}
\begin{tikzpicture}[node distance=6cm, auto]
  \node (A) {$X_{i} \hat{\otimes} k$};
  \node (B) [below=0.5cm of A] {$X_{i} \hat{\otimes} P(W)$};
  \node (C) [right=2cm of A] {$X_{i}$};
  \node (D) [below=0.5cm of C] {$Y_{j_{i}}$};
  \draw[->] (A) to node [swap]{$\text{Id} \otimes f_{x}$} (B);
  \draw[->] (A) to node {$\sim$} (C);
  \draw[->] (B) to node {} (D);
  \draw[->] (C) to node {$\mathscr{V}_{i}(f_{x})$} (D);
\end{tikzpicture}
\end{center}
commutes, and hence so does the diagram
\begin{center}
\begin{tikzpicture}[node distance=6cm, auto]
  \node (A) {$\mathcal{V}(k) \hat{\otimes} k$};
  \node (B) [below=0.5cm of A] {$\mathcal{V}(k) \hat{\otimes} P(W)$};
  \node (C) [right=2cm of A] {$\mathcal{V}(k)$};
  \node (D) [below=0.5cm of C] {$\mathcal{V}(P(W))$.};
  \draw[->] (A) to node [swap]{$\text{Id} \otimes f_{x}$} (B);
  \draw[->] (A) to node {$\sim$} (C);
  \draw[->] (B) to node {$\sim$} (D);
  \draw[->] (C) to node {$\mathscr{V}(f_{x})$} (D);
\end{tikzpicture}
\end{center}
Thus the diagram
\begin{center}
\begin{tikzpicture}[node distance=6cm, auto]
  \node (A) {$\mathcal{V}(k) \hat{\otimes} P(W')$};
  \node (B) [below=0.5cm of A] {$\mathcal{V}(k) \hat{\otimes} P(W)$};
  \node (C) [right=1cm of A] {$\mathcal{V}(P(W'))$};
  \node (D) [below=0.5cm of C] {$\mathcal{V}(P(W))$};
  \draw[->] (A) to node [swap]{$\text{Id} \otimes f$} (B);
  \draw[->] (A) to node {$\sim$} (C);
  \draw[->] (B) to node {$\sim$} (D);
  \draw[->] (C) to node {$\mathscr{V}(f)$} (D);
\end{tikzpicture}
\end{center}
must also commute. From this we have that
$$\begin{array}{rcl}
\mathscr{V}(W) &\cong& \mathscr{V}(\text{Coker}(P(W') \xrightarrow{f} P(W)))\\
&\cong& \text{Coker}(\mathscr{V}(P(W')) \xrightarrow{\mathscr{V}(f)} \mathscr{V}(P(W)))\\
&\cong& \text{Coker}(\mathscr{V}(k) \hat{\otimes}P(W') \xrightarrow{\text{Id} \otimes f} \mathscr{V}(k) \hat{\otimes}P(W))\\
&\cong& \mathscr{V}(k) \hat{\otimes}\text{Coker}(P(W') \xrightarrow{f} P(W))\\
&\cong& \mathscr{V}(k) \hat{\otimes} W\\
\end{array}$$
for any Banach space $W$. Since any IndBanach space can be written as a colimit of Banach spaces, and since both $\mathscr{V}$ and $\mathscr{V}(k)\hat{\otimes}-$ are cocontinuous, $\mathscr{V}$ is isomorphic to the functor $V \hat{\otimes} -$ for $V=\mathscr{V}(k)$.
\end{proof}

\begin{theorem}
\label{IndBanachTannaka}
Let $\mathcal{C}$ be a locally presentable, quasi-abelian category, enriched over $\text{IndBan}_{k}$, equipped with a fibre functor $F:\mathcal{C} \rightarrow \text{IndBan}_{k}$ as in Definition 2.1. Assume further that $T=FG$ commutes with $l^{1}$, as in Definition \ref{CommutingWithl1}, for some left adjoint $G$ to $F$. Then there exists an algebra $\mathscr{A}$ in $\text{IndBan}_{k}$ such that $\mathcal{C}$ is equivalent to the category of left $\mathscr{A}$ modules in $\text{IndBan}_{k}$. 
\end{theorem}
\begin{proof}
By Lemma \ref{FibreFunctorsMonadic}, $\mathcal{C}$ is equivalent to the category of alegbras of $T$ in $\text{IndBan}_{k}$. By Lemma \ref{FibreFunctorMonadCocontinuous}, Lemma \ref{TensorFunctorClassification} and our assumption that $T$ commutes with $l^{1}$, $T$ is isomorphic to $\mathscr{A} \hat{\otimes} -$ for $\mathscr{A} = T(k)$. Then the fact that $T$ is a monad is equivalent to $\mathscr{A}$ being an algebra, and the category of $T$ algebras in $\text{IndBan}_{k}$ is then just the category of $\mathscr{A}$ modules. 
\end{proof}

\begin{defn}
Let $\mathcal{C}$ be a category enriched over $\text{IndBan}_{k}$. We will say that $\mathcal{C}$ has \emph{constant contracting coproducts} if, for each set $S$, there is a functor
$$\coprod\nolimits^{\leq 1}_{S}:\mathcal{C} \rightarrow \mathcal{C}$$
and, for each map of sets $S' \rightarrow S$, there is a natural transformation
$$\coprod\nolimits^{\leq 1}_{S} \Rightarrow \coprod\nolimits^{\leq 1}_{S'}$$
such that
\begin{itemize}
\item[i)]$\underline{\text{Hom}}_{\mathcal{C}}(\coprod^{\leq 1}_{S} X,Y) \cong \prod^{\leq 1}_{S} \underline{\text{Hom}}_{\mathcal{C}}(X,Y)$ for all $X$ and $Y$ in $\mathcal{C}$;
\item[ii)]the assignment $S \mapsto \coprod\nolimits^{\leq 1}_{S}$ is contravariantly functorial.
\end{itemize}
By property (i), if such functors exist then they exist uniquely. We will say that a functor $F:\mathcal{C}\rightarrow \mathcal{C}'$ between categories with constant contracting coproducts commutes with constant contracting coproducts if we have a collection of natural isomorphisms $F \circ \coprod\nolimits^{\leq 1}_{S} \cong \coprod\nolimits^{\leq 1}_{S} \circ F$ compatible with the functoriality in $S$.
\end{defn}

\begin{remark}
In the case where $\mathcal{C}=\text{IndBan}_{k}$, the functor $\coprod\nolimits^{\leq 1}_{S}$ is the composition
$$\text{IndBan}_{k} \xrightarrow{\Delta^{S}} \text{IndBan}_{k}^{S,\text{bd}} \xrightarrow{\coprod\nolimits^{\leq 1}_{S}} \text{IndBan}_{k}.$$
\end{remark}

\begin{corollary}
\label{IndBanachTannakaCorollary}
Suppose we have a category $\mathcal{C}$ with constant contracting coproducts that is fibred over $\text{IndBan}_{k}$ as defined above. Suppose further that the fiber functor $F$ commutes with constant contracting coproducts. Then there exists an algebra $\mathscr{A}$ in $\text{IndBan}_{k}$ such that $\mathcal{C}$ is equivalent to the category of left $\mathscr{A}$ modules in $\text{IndBan}_{k}$.\\
\end{corollary}

\begin{proof}
Let $G$ denote the left adjoint to $F$, which exists by Lemma \ref{FibreFunctorsMonadic}. We have that
$$\begin{array}{rcl}
\underline{\text{Hom}}(\coprod_{S}^{\leq 1} G(X),Y) &\cong& \prod_{S}^{\leq 1} \underline{\text{Hom}}(G(X),Y)\\
&\cong& \prod_{S}^{\leq 1} \underline{\text{Hom}}(X,F(Y))\\
&\cong& \underline{\text{Hom}}(\coprod_{S}^{\leq 1}X,F(Y))\\
&\cong& \underline{\text{Hom}}(G(\coprod_{S}^{\leq 1}X),Y)
\end{array}$$
for all $X$ in $\text{IndBan}_{k}$ and $Y$ in $\mathcal{C}$, hence $G(\coprod_{S}^{\leq 1}X) \cong \coprod_{S}^{\leq 1}G(X)$ naturally for all $X$ in $\text{IndBan}_{k}$. The result then follows from Theorem \ref{IndBanachTannaka}.
\end{proof}

We may, in fact, give an alternate and perhaps more explicit description of the algebra $\mathscr{A}$ from Theorem \ref{IndBanachTannaka} and Corollary \ref{IndBanachTannakaCorollary}.

\begin{defn}
Let $\mathscr{F}:\mathscr{C} \rightarrow \text{IndBan}_{k}$ be a functor. As $\text{IndBan}_{k}$ is closed, we may define the \emph{internal natural transformations} $\underline{\text{Hom}}(\mathscr{F},\mathscr{F})$ from $\mathscr{F}$ to itself as the end
$$\int_{V \in \mathcal{C}} \underline{\text{Hom}}(\mathscr{F}V,\mathscr{F}V) = \text{eq}\left(\prod_{V \in \mathcal{C}} \underline{\text{Hom}}(\mathscr{F}V,\mathscr{F}V) \rightrightarrows \prod_{V\rightarrow V'} \underline{\text{Hom}}(\mathscr{F}V,\mathscr{F}V') \right).$$
The maps $k \rightarrow \underline{\text{Hom}}(\mathscr{F}V,\mathscr{F}V)$ picking out the identity in $\underline{\text{Hom}}(\mathscr{F}V,\mathscr{F}V)$ induce a unit map
$$k \rightarrow \int_{V \in \mathcal{C}} \underline{\text{Hom}}(\mathscr{F}V,\mathscr{F}V) = \underline{\text{Hom}}(\mathscr{F},\mathscr{F})$$
and the compositions $\underline{\text{Hom}}(\mathscr{F}V,\mathscr{F}V) \hat{\otimes} \underline{\text{Hom}}(\mathscr{F}V,\mathscr{F}V) \rightarrow \underline{\text{Hom}}(\mathscr{F}V,\mathscr{F}V)$ induce a multiplication
$$
\left( \int_{V \in \mathcal{C}} \underline{\text{Hom}}(\mathscr{F}V,\mathscr{F}V) \right) \hat{\otimes} \left( \int_{V \in \mathcal{C}} \underline{\text{Hom}}(\mathscr{F}V,\mathscr{F}V) \right) \rightarrow \int_{V \in \mathcal{C}} \underline{\text{Hom}}(\mathscr{F}V,\mathscr{F}V)
$$
from which we give $\underline{\text{Hom}}(\mathscr{F},\mathscr{F})$ the expected IndBanach algebra structure.
\end{defn}

\begin{prop}
Let $\mathscr{A}$ be an IndBanach algebra, let $\mathcal{C}$ be the category of its IndBanach modules and let $F$ be the forgetful functor to IndBanach spaces. Then $\mathscr{A} \cong \underline{\text{Hom}}(F,F)$ as IndBanach algebras.
\end{prop}

\begin{proof}
$\mathscr{A}$ naturally gives an object of $\mathcal{C}$, and $F \cong \underline{\text{Hom}}(\mathscr{A},-)$. So, by the enriched Yoneda Lemma (see Section 2.4 of \cite{BCoECT}), $\mathscr{A} \cong \underline{\text{Hom}}(F,F)$ cannonically.  It is clear from construction that this is an isomorphism of IndBanach algebras.
\end{proof}

\begin{remark}
\label{IndBanachMonoidalFunctors}
Suppose $\mathcal{C}$ is the category of IndBanach modules over an IndBanach algebra $\mathscr{A}$. Let $F$ denote the forgetful functor to $\text{IndBan}_{k}$, $G$ its left adjoint, and $T=FG \cong \mathscr{A} \hat{\otimes} -$. Moerdijk proves in \cite{MoTC} that monoidal structures on $\mathcal{C}$ for which $F$ is strong monoidal correspond to comonoidal structures on $T$, which in turn correspond to coalgebra structures on $\mathscr{A}$. For any given monoidal structure on $\mathcal{C}$ with $F$ strong monoidal, the counit of the adjunction gives us a morphism $T(k) \rightarrow k$. The image of $\eta_{V} \hat{\otimes} \eta_{W}$ under
$$\begin{array}{rcl}
\text{Hom}(A \hat{\otimes} B, FGA \hat{\otimes} FGB) &\cong& \text{Hom}(A \hat{\otimes} B, F(GA \hat{\otimes} GB))\\
&\cong& \text{Hom}(G(A \hat{\otimes} B), GA \hat{\otimes} GB).
\end{array}$$
gives a natural transfromation $G(- \hat{\otimes} -) \Rightarrow G(-) \hat{\otimes} G(-)$. Then the composite $T(- \hat{\otimes} -) \Rightarrow F(G(-) \hat{\otimes} G(-)) \cong T(-) \hat{\otimes} T(-)$ makes $T$ comonoidal. This gives $\mathscr{A}$ a comultiplication compatible with its multiplication, from which the monoidal structure of $\mathcal{C}$ comes.
\end{remark}

\subsection{IndBanach comodules of IndBanach coalgebras}
\

Classical Tannaka-Krein duality asks when a category $\mathcal{C}$ is a category of comodules over a coalgebra, which we aim to provide an analytic analogue of here.

\begin{defn}
Let $\mathcal{C}$ be a locally presentable, quasi-abelian category, enriched over $\text{IndBan}_{k}$, and let $F:\mathcal{C} \rightarrow \text{IndBan}_{k}$ be an enriched functor. We say that $F$ is a \emph{co-fibre functor} if it is cocontinuous, strongly exact, faithful and reflects strict morphisms.
\end{defn}

\begin{lem}
\label{CoFibreFunctorsComonadic}
Let $\mathcal{C}$ be a locally presentable, quasi-abelian category, enriched over $\text{IndBan}_{k}$, and let $F:\mathcal{C} \rightarrow \text{IndBan}_{k}$ be a co-fibre functor over $\text{IndBan}_{k}$. Then $F$ satisfies the dual conditions of Barr-Beck (Theorem \ref{DualBarrBeck}), so $\mathcal{C}$ is equivalent to the category of coalgebras of a comonadic functor $U$ in $\text{IndBan}_{k}$.
\end{lem}

\begin{proof}
The proof is entirely similar to that of Lemma \ref{FibreFunctorsMonadic}.
\end{proof}

\begin{remark}
Since $G$ is a right adjoint, it is not necessarily true that $G$ or $U$ is cocontinuous.
\end{remark}

\begin{theorem}
\label{CoIndBanachTannaka}
Let $\mathcal{C}$ be a locally presentable, $k$-linear, quasi-abelian category, equipped with a co-fibre functor $F: \mathcal{C} \rightarrow \text{IndBan}_{k}$. Assume further that $U=FG$ is cocontinuous and commutes with $l^{1}$, where $G$ is some right adjoint to $F$. Then there exists a coalgebra $\mathscr{B}$ in $\text{IndBan}_{k}$ such that $\mathcal{C}$ is equivalent to the category of left $\mathscr{B}$ comodules in $\text{IndBan}_{k}$. 
\end{theorem}

\begin{proof}
The proof is analogous to that of Theorem \ref{IndBanachTannaka}.
\end{proof}

\begin{remark}
At first sight this result is less satisfying than Theorem \ref{IndBanachTannaka} or Corollary \ref{IndBanachTannakaCorollary}, as $U$ is not automatically cocontinuous and the assumption that it commutes with $l^{1}$ is not automatic from a good notion of contracting coproducts. However, in applications, the adjoint $G$, and hence the comonad $U$, can often be described explicitly and checked for (contracting) cocontinuity.
\end{remark}

\begin{remark}
Let $\mathcal{C}$ be the category of IndBanach comodules over an IndBanach coalgebra $\mathscr{B}$, and let us denote by $F$ the forgetful functor to $\text{IndBan}_{k}$, $G$ its right adjoint, and $U=FG \cong \mathscr{B} \hat{\otimes} -$. As before, monoidal structures on $\mathcal{C}$ for which $F$ is strong monoidal were shown by Moerdijk in \cite{MoTC} to correspond directly to monoidal structures on $U$, which in turn correspond to algebra structures on $\mathscr{B}$. This correspondence is dual to the one outlined in the final Remark of Subsection \ref{IndBanachMonoidalFunctors}.
\end{remark}

\subsection{Simultaneous modules and comodules}
\

In the case where we have both left and right adjoints $G$ and $G'$ as described in Theorems \ref{IndBanachTannaka} and \ref{CoIndBanachTannaka}, we relate $\mathscr{A}=G(k)$ and $\mathscr{B}=G'(k)$ as follows.

\begin{prop}
$\mathscr{A}$ is dualisable with dual $\mathscr{B}$ in $\text{IndBan}_{k}$.
\end{prop}
\begin{proof}
The adjunction gives an adjunction between $T=GF$ and $U=FG'$
$$\text{Hom}(TV,W) \cong \text{Hom}(GV,G'W) \cong  \text{Hom}(V,UW).$$
Then the unit and counit of this adjunction give the duality.
\end{proof}

\begin{remark}
Conversely, suppose that $\mathscr{A}$ is a dualisable IndBanach algebra with dual $\mathscr{B}$. Then $\mathscr{B}$ forms an IndBanach coalgebra, and there is an adjunction as above between the functors $T=\mathscr{A} \hat{\otimes} -$ and $U=\mathscr{B} \hat{\otimes} -$. It then follows that the category of IndBanach $\mathscr{A}$ modules and IndBanach $\mathscr{B}$ comodules are equivalent in a way compatible with the forgetful functor.
\end{remark}

\section{Examples}

We now present some examples to highlight the possible applications of this theory.

\subsection{Comodules of a Banach coalgebra}
\begin{prop}
\label{LocallyBanachComodules}
Let $B$ be a Banach coalgebra, viewed as an IndBanach space, and let $M$ be an IndBanach $B$-comodule. Then $M$ is isomorphic to a colimit of Banach comodules of $B$.
\end{prop}
\begin{proof}
Let $\mathcal{C}$ be the Ind completion of the category of Banach $B$-comodules. The forgetful functor from the category of Banach $B$-comodules to $\text{Ban}_{k}$ induces a cocontinuous, strongly exact, faithful functor $F:\mathcal{C} \rightarrow \text{IndBan}_{k}$ that reflects strict morphisms. Hence $F$ is a co-fibre functor and $\mathcal{C}$ is equivalent to the category of coalgebras over a comonad $U$. The right adjoint of the forgetful functor $B \hat{\otimes} -$ from $\text{Ban}_{k}$ to Banach $B$-comodules induces a right adjoint $G$ to $F$, and $FG$ is isomorphic to the functor $B \hat{\otimes} -$. So it follows that $\mathcal{C}$ is equivalent to the category of $B$-comodules in $\text{IndBan}_{k}$, from which the proposition follows.
\end{proof}

\subsection{Analytic gradings}
\label{AnalyticGradings}
\

This example is motivated by the prospect of defining analytic analogues of quantum groups. In constructing the positive part of the quantum group through Nicholls algebras, one works with graded vector spaces. The following gives an analytic analogue of such a grading.\\

\begin{defn}
Let $\text{Gr}_{\mathbb{Z}}\text{IndBan}_{k}$ be the category of IndBanach spaces of the form $\coprod_{n \in \mathbb{Z}}^{\leq 1} M(n)$ with morphisms that preserve this grading, that is
$$\underline{\text{Hom}}_{\text{Gr}_{\mathbb{Z}}}(\coprod\nolimits_{n \in \mathbb{Z}}^{\leq 1} M(n),\coprod\nolimits_{n' \in \mathbb{Z}}^{\leq 1} M'(n'))=\prod\nolimits_{n \in \mathbb{Z}}^{\leq 1} \underline{\text{Hom}}(M(n),M'(n)).$$
Let $F$ be the forgetful functor to $\text{IndBan}_{k}$ which maps morphisms via the usual map
$$
\begin{array}{rcl}
\prod\nolimits_{n \in \mathbb{Z}}^{\leq 1} \underline{\text{Hom}}(M(n),M'(n)) &\rightarrow& \prod\nolimits_{n \in \mathbb{Z}}^{\leq 1}\coprod\nolimits_{n' \in \mathbb{Z}}^{\leq 1} \underline{\text{Hom}}(M(n),M'(n'))\\
&=&\underline{\text{Hom}}(\coprod\nolimits_{n \in \mathbb{Z}}^{\leq 1} M(n),\coprod\nolimits_{n' \in \mathbb{Z}}^{\leq 1} M'(n')).
\end{array}$$
$\text{Gr}_{\mathbb{Z}}\text{IndBan}_{k}$ is monoidal, with tensor product
$$\left(\coprod\nolimits_{n \in \mathbb{Z}}^{\leq 1} M(n)\right) \hat{\otimes} \left(\coprod\nolimits_{n' \in \mathbb{Z}}^{\leq 1} M'(n')\right) = \coprod\nolimits_{N \in \mathbb{Z}}^{\leq 1} \left(\coprod\nolimits_{n+n'=N}^{\leq 1}M(n) \hat{\otimes} M'(n')\right).$$
\end{defn}

\begin{prop}
\label{ZGradedIndBanachSpaces}
$\text{Gr}_{\mathbb{Z}}\text{IndBan}_{k}$ is equivalent to the monoidal category of $\mathscr{B}$ comodules in $\text{IndBan}_{k}$, where $\mathscr{B}$ is the bialgebra $\coprod_{n \in \mathbb{Z}}^{\leq 1}k \cdot t^{n}$. Here, $\mathscr{B}$ has the comultiplication $t^{n}\mapsto t^{n} \otimes t^{n}$, with counit $t^{n} \mapsto 1$, and multiplication $t^{n} \cdot t^{n'}=t^{n+n'}$, with unit $t^{0}$.
\end{prop}
\begin{proof}
Since
$$\underline{\text{Hom}}(F(\coprod\nolimits_{n \in \mathbb{Z}}^{\leq 1}M(n)),X)=\prod\nolimits_{n \in \mathbb{Z}}^{\leq 1}\underline{\text{Hom}}(M(n),X),$$
we see that $F$ is left adjoint to the functor $G:\text{IndBan}_{k} \rightarrow \text{Gr}_{\mathbb{Z}}\text{IndBan}_{k}$ that takes $X$ to the contracting coproduct $\coprod_{\mathbb{Z}}^{\leq 1}X$. Then $U=FG$ is cocontinuous and commutes with contracting colimits, so is isomorphic to the functor $\mathscr{B} \hat{\otimes} -$. It is clear that the monoidal comonadic structure on $U$ induces the above bialgebra structure on $\mathscr{B}$.
\end{proof}

\begin{remark}
The bialgebra $\mathscr{B}$ can be thought of as a completion of $k[t,t^{-1}]$, the bialgebra of analytic functions on the unit circle in $k$, whose vector space comodules are $\mathbb{Z}$-graded vector spaces.
\end{remark}

\begin{remark}
In fact, if $\Gamma$ is any discrete group and $\text{Gr}_{\Gamma}\text{IndBan}_{k}$ is the category of IndBanach spaces with an analytic $\Gamma$ grading, $M=\coprod_{g \in \Gamma}^{\leq 1} M(g)$, then a similar argument to the above shows the following.
\end{remark}

\begin{prop}
\label{GammaGradedIndBanachSpaces}
The analogously defined category $\text{Gr}_{\Gamma}\text{IndBan}_{k}$ is equivalent to the monoidal category of comodules of the bialgebra $\coprod_{g \in \Gamma}^{\leq 1}k \cdot t^{g}$. Here we have comultiplication $t^{g} \mapsto t^{g} \otimes t^{g}$, counit $t^{g} \mapsto 1$, multiplication $t^{g} \cdot t^{h} = t^{gh}$ and unit $t^{e}$.
\end{prop}

\begin{remark}
If we take $\Gamma=\mathbb{Z}^{n}$ we obtain a completion of $k[t_{1},t_{1}^{-1},\ldots,t_{n},t_{n}^{-1}]$, the coalgebra of analytic functions on the unit sphere in $k^{n}$.
\end{remark}

\begin{remark}
Note that, in the above, the forgetful functor is not continuous. The product of a collection $(\coprod_{n \in \mathbb{Z}}^{\leq 1} N(n,i))_{i \in I}$ in $\text{Gr}_{\mathbb{Z}}\text{IndBan}_{k}$ is $\coprod_{n \in \mathbb{Z}}^{\leq 1} \prod_{i \in I} N(n,i)$ since
\begin{align*}
&\text{Hom}_{\text{Gr}_{\mathbb{Z}}}(\coprod\nolimits_{m \in \mathbb{Z}}^{\leq 1} M(m),\coprod\nolimits_{n \in \mathbb{Z}}^{\leq 1} \prod_{i \in I} N(n,i))\\
& \qquad =\prod\nolimits_{n \in \mathbb{Z}}^{\leq 1} \prod_{i \in I}\text{Hom}(M(n),N(n,i))\\
& \qquad =\prod_{i \in I}\prod\nolimits_{n \in \mathbb{Z}}^{\leq 1} \text{Hom}(M(n),N(n,i))\\
& \qquad =\prod_{i \in I}\text{Hom}_{\text{Gr}_{\mathbb{Z}}}(\coprod\nolimits_{m \in \mathbb{Z}}^{\leq 1} M(m),\coprod\nolimits_{n \in \mathbb{Z}}^{\leq 1} N(n,i)),
\end{align*}
but it is not necessarily true that products commute with contracting coproducts in $\text{IndBan}_{k}$.
\end{remark}

\subsection{Gradings arising from strictly convergent and overconvergent powerseries on the unit polydisk}
\label{DaggerGradings}
\

In the previous example, we showed that analytically $\mathbb{Z}^{n}$-graded IndBanach spaces are comodules over the bialgebra of analytic functions on the unit sphere in $k^{n}$. There are, of course, other spaces of analytic functions, and these give rise to other analytic gradings.

\begin{defn}
Let $\text{Gr}_{\mathbb{N}^{N}}\text{IndBan}_{k}$ be the category whose objects are IndBanach spaces of the form $\coprod_{\underline{n} \in \mathbb{N}^{N}}^{\leq 1}M(\underline{n})$ with morphisms that respect the grading
$$\text{Hom}_{\text{Gr}_{\mathbb{N}^{n}}}(\coprod\nolimits_{\underline{n} \in \mathbb{N}^{N}}^{\leq 1}M(\underline{n}),\coprod\nolimits_{\underline{n}' \in \mathbb{N}^{N}}^{\leq 1}M'(\underline{n}'))= \prod\nolimits_{\underline{n} \in \mathbb{N}^{N}}^{\leq 1}\text{Hom}(M(\underline{n}),M'(\underline{n})).$$
Let us denote by $F$ the forgetful functor to $\text{IndBan}_{k}$. $\text{Gr}_{\underline{r}}\text{IndBan}_{k}$ is monoidal, with
$$\left( \coprod\nolimits_{\underline{n}}^{\leq 1}M(\underline{n}) \right) \hat{\otimes} \left( \coprod\nolimits_{\underline{n}' }^{\leq 1}M'(\underline{n}') \right)=\coprod\nolimits_{\underline{n} }^{\leq 1} \left( \coprod\nolimits_{\underline{m}+\underline{m}'=\underline{n}}^{\leq 1}M(\underline{m})\hat{\otimes} M'(\underline{m}')\right).$$
\end{defn}

\begin{prop}
The category $\text{Gr}_{\mathbb{N}^{N}}\text{IndBan}_{k}$ is equivalent to the category of $k\{\underline{t}\}=k\{t_{1},\ldots,t_{N}\}:=\coprod_{\underline{n} \in \mathbb{N}^{N}}^{\leq 1} k \cdot \underline{t}^{\underline{n}}$ comodules, where the comultiplication maps $\underline{t}^{\underline{n}} \mapsto \underline{t}^{\underline{n}} \otimes \underline{t}^{\underline{n}}$ and the counit is $\underline{t}^{\underline{n}} \mapsto 1$, and the multiplication maps $\underline{t}^{\underline{m}} \otimes \underline{t}^{\underline{n}} \mapsto  \underline{t}^{\underline{m}+ \underline{n}}$ with unit $\underline{t}^{\underline{0}}$.
\end{prop}

\begin{proof}
This is just a variation of Proposition \ref{ZGradedIndBanachSpaces}.
\end{proof}

\begin{remark}
This is the bialgebra of strictly convergent powerseries on the polydisk of radius $1$, $\{\underline{a}=(a_{1},..,a_{N}) \in k^{N} \mid |a_{i}| \leq 1\}$. Note that strictly convergent powerseries on a polydisk of polyradius $\underline{r}$ does not have a well defined comultiplication unless all $r_{i} \leq 1$, and the counit is only well defined if all $r_{i} \geq 1$, hence we are restricted to the unit polydisk.
\end{remark}

\begin{defn}
Let $\text{Gr}_{\mathbb{N}^{N}}^{\dagger,1}\text{IndBan}_{k}$ be the category whose objects are IndBanach spaces of the form $M=\text{"colim"}_{\underline{r} > 1}\coprod_{\underline{n} \in \mathbb{N}^{N}}^{\leq 1} M(\underline{n})_{\underline{r}^{n}}$, with morphisms
$$\text{Hom}_{\text{Gr}_{\mathbb{N}^{N}}^{\dagger,1}}(M,M')=\text{lim}_{\underline{r}<1}\prod\nolimits_{\underline{n} \in \mathbb{N}^{n}}^{\leq 1} \text{Hom}(M(\underline{n}),M'(\underline{n}))_{\underline{r}^{\underline{n}}},$$
for $M=\text{"colim"}_{\underline{r} > 1}\coprod\nolimits_{\underline{n} \in \mathbb{N}^{N}}^{\leq 1} M(\underline{n})_{\underline{r}^{n}}$ and $M'=\text{"colim"}_{\underline{r}' > 1}\coprod\nolimits_{\underline{n'} \in \mathbb{N}^{N}}^{\leq 1} M'(\underline{n'})_{\underline{r'}^{n'}}$. Here, colimits and limits are taken over polyradii $\underline{r}=(r_{1},..,r_{N})$ with $1 < r_{i}$ and $1 > r_{i}$ respectively for $i=1,..,N$, and for an IndBanach space $V=\text{"colim"}_{i \in I}V_{i}$ and for $\lambda \in \mathbb{R}_{>0}$, we use the notation $V_{\lambda}$ for the IndBanach space $V_{\lambda}=\text{"colim"}_{i \in I}(V_{i})_{\lambda}$, where $(V_{i})_{\lambda}$ is the Banach space whose underlying vector space is that of $V_{i}$ but with the norm scaled by $\lambda$. The category $\text{Gr}_{\mathbb{N}^{N}}^{\dagger,1}\text{IndBan}_{k}$ is monoidal, with
$$M \hat{\otimes} M' =\text{"colim"}_{\underline{r}>1}\coprod\nolimits_{\underline{n} }^{\leq 1} \left( \coprod\nolimits_{\underline{m}+\underline{m}'=\underline{n}}^{\leq 1}M(\underline{m})\hat{\otimes} M'(\underline{m}')\right)_{\underline{r}^{\underline{n}}}.$$
\end{defn}

\begin{prop}
\label{Dagger1Grading}
The category $\text{Gr}_{\mathbb{N}^{N}}^{\dagger,1}\text{IndBan}_{k}$ is equivalent to the monoidal category of $k\{\underline{t}\}^{\dagger}:=\text{"colim"}_{\underline{r} > 1} \coprod\nolimits_{\underline{n} \in \mathbb{N}^{n}}^{\leq 1}k_{\underline{r}^{\underline{n}}}$ comodules. The algebra structure comes from that of each $k\{\frac{\underline{t}}{\underline{r}}\}=\coprod\nolimits_{\underline{n} \in \mathbb{N}^{n}}^{\leq 1}k_{\underline{r}^{\underline{n}}}$, whilst the counit and comultiplication are induced by the maps
$$
\begin{array}{cccc}
k\{\frac{\underline{t}}{\underline{r}}\} \rightarrow k, & \underline{t}^{\underline{n}} \mapsto 1, & k\{\frac{\underline{t}}{\underline{r}^{2}}\} \rightarrow k\{\frac{\underline{t}}{\underline{r}}\} \hat{\otimes} k\{\frac{\underline{t}}{\underline{r}}\}, & \underline{t}^{\underline{n}} \mapsto \underline{t}^{\underline{n}} \otimes \underline{t}^{\underline{n}}.
\end{array}
$$
\end{prop}

\begin{proof}
For each IndBanach space $X$
\begin{align*}
&\text{Hom}_{\text{Gr}_{\mathbb{N}^{N}}^{\dagger}}(\text{"colim"}_{\underline{r} > 1}\coprod\nolimits_{\underline{n} \in \mathbb{N}^{N}}^{\leq 1}M(\underline{n})_{\underline{r}^{\underline{n}}},\text{"colim"}_{\underline{r}' > \underline{\rho}} \coprod\nolimits_{\underline{n} \in \mathbb{N}^{n}}^{\leq 1}X_{\underline{r}'^{\underline{n}}})\\
& ~ =\text{lim}_{\underline{r}<1}\prod\nolimits_{\underline{n} \in \mathbb{N}^{n}}^{\leq 1} \text{Hom}(M(\underline{n}),X)_{\underline{r}^{\underline{n}}}\\
& ~ =\text{lim}_{\underline{r}>1}\prod\nolimits_{\underline{n} \in \mathbb{N}^{n}}^{\leq 1} \text{Hom}(M(\underline{n}),X)_{(1/\underline{r})^{\underline{n}}}\\
& ~ =\text{lim}_{\underline{r}>1}\prod\nolimits_{\underline{n} \in \mathbb{N}^{n}}^{\leq 1} \text{Hom}(M(\underline{n})_{\underline{r}^{\underline{n}}},X)\\
& ~ =\text{Hom}(\text{"colim"}_{\underline{r}>1}\coprod\nolimits_{\underline{n} \in \mathbb{N}^{n}}^{\leq 1} M(\underline{n})_{\underline{r}^{\underline{n}}},X)
\end{align*}
and so $X \mapsto \text{"colim"}_{\underline{r} > 1} \coprod\nolimits_{\underline{n} \in \mathbb{N}^{n}}^{\leq 1}X_{\underline{r}^{\underline{n}}}$ is right adjoint to the forgetful functor. The associated comonad is the isomorphic to $k\{\underline{t}\}^{\dagger} \hat{\otimes} -$. The monoidal structure on $\text{Gr}_{\mathbb{N}^{N}}^{\dagger}\text{IndBan}_{k}$ gives $k\{\underline{t}\}^{\dagger}$ the described bialgebra structure.
\end{proof}

\begin{remark}
$k\{\underline{t}\}^{\dagger}$ is referred to as the bialgebra of \emph{overconvergent powerseries} on on the polydisk of radius $1$. For similar reasons to the case of strictly convergent powerseries, we are restricted on our choice of polyradius. Alongside the previous example of radius $1$, we also have the following at radius $0$, where we consider germs of analytic functions at $0$.
\end{remark}

\begin{defn}
Let $\text{Gr}_{\mathbb{N}^{N}}^{\dagger,0}\text{IndBan}_{k}$ be the category whose objects are IndBanach spaces of the form $M=\text{"colim"}_{\underline{r} > 0}\coprod_{\underline{n} \in \mathbb{N}^{N}}^{\leq 1} M(\underline{n})_{\underline{r}^{n}}$, with morphisms
$$\text{Hom}_{\text{Gr}_{\mathbb{N}^{N}}^{\dagger,0}}(M,M')=\text{lim}_{\underline{r}>0}\prod\nolimits_{\underline{n} \in \mathbb{N}^{n}}^{\leq 1} \text{Hom}(M(\underline{n}),M'(\underline{n}))_{\underline{r}^{\underline{n}}},$$
for $M=\text{"colim"}_{\underline{r} > 0}\coprod\nolimits_{\underline{n} \in \mathbb{N}^{N}}^{\leq 1} M(\underline{n})_{\underline{r}^{n}}$ and $M'=\text{"colim"}_{\underline{r}' > 0}\coprod\nolimits_{\underline{n'} \in \mathbb{N}^{N}}^{\leq 1} M'(\underline{n'})_{\underline{r'}^{n'}}$.  As before, the category $\text{Gr}_{\mathbb{N}^{N}}^{\dagger,0}\text{IndBan}_{k}$ is monoidal.
\end{defn}

\begin{prop}
The category $\text{Gr}_{\mathbb{N}^{N}}^{\dagger,0}\text{IndBan}_{k}$ is equivalent to the monoidal category of $k\{\frac{\underline{t}}{0}\}^{\dagger}:=\text{"colim"}_{\underline{r} > 0} \coprod\nolimits_{\underline{n} \in \mathbb{N}^{n}}^{\leq 1}k_{\underline{r}^{\underline{n}}}$ comodules.
\end{prop}

\begin{proof}
This follows as in the proof of Proposition \ref{Dagger1Grading}.
\end{proof}

\subsection{Non-example: Contracting products}
\

Let $\mathcal{C}$ be the category of IndBanach spaces of the form $\prod_{n \in \mathbb{Z}}^{\leq 1}M(n)$ with morphisms similar to $\text{Gr}_{\mathbb{Z}}\text{IndBan}_{K}$,
$$\text{Hom}_{\mathcal{C}}(\prod\nolimits_{n \in \mathbb{Z}}^{\leq 1} M(n),\prod\nolimits_{n' \in \mathbb{Z}}^{\leq 1} M'(n'))=\prod\nolimits_{n \in \mathbb{Z}}^{\leq 1} \text{Hom}(M(n),M'(n)),$$
and again let $F$ be the forgetful functor to $\text{IndBan}_{k}$. Then as
$$\text{Hom}(X,F(\prod\nolimits_{n \in \mathbb{Z}}^{\leq 1}M(n))) \cong \prod\nolimits_{n \in \mathbb{Z}}^{\leq 1} \text{Hom}(X,M(n))$$
we see that $F$ has as left adjoint the functor $G':X \mapsto \prod_{n \in \mathbb{Z}}^{\leq 1}X$. However $T=FG'$ does not commute with contracting coproducts, and so is not isomorphic to taking the tensor product with an IndBanach algebra.

\subsection{Representations of discrete groups}
\begin{defn}
Consider a discrete group $\Gamma$, and let $\Gamma\text{-}\text{IndBan}_{k}$ be the category of representations of $\Gamma$ on IndBanach spaces. This has the obvious forgetful functor $F$ to $\text{IndBan}_{k}$ forgetting the action of $\Gamma$. With the diagonal action of $\Gamma$, $\mathcal{C}$ is monoidal and $F$ is strong monoidal.
\end{defn}

\begin{lem}
\label{DiscreteGroupsAdjunction}
$F$ has a left adjoint $G:X \mapsto \coprod_{g \in \Gamma} X$ where $h \in \Gamma$ acts on $GX$ by mapping the copy of $X$ indexed by $g$ isomorphically to the copy indexed by $hg$.
\end{lem}

\begin{proof}
The isomorphism $\text{Hom}(X, FY) \cong \text{Hom}_{\Gamma}(\coprod_{g \in \Gamma}X,Y)$, for $X$ in $\text{IndBan}_{k}$ and $Y$ in $\Gamma\text{-}\text{IndBan}_{k}$, that gives this adjunction takes $f:X \rightarrow Y$ to the morphism defined on the copy of $X$ indexed by $g$ as $X \overset{f}{\rightarrow} Y \overset{g \cdot}{\rightarrow} Y$. The inverse to this map just restricts a morphism $\coprod_{g \in \Gamma} X \rightarrow Y$ to the copy of $X$ indexed by the identity $1 \in G$.
\end{proof}

\begin{prop}
\label{RecoveringBornologicalGroupAlgebra}
$\Gamma\text{-}\text{IndBan}_{k}$ is equivalent to the monoidal category of $\mathscr{A} = \coprod_{g \in \Gamma}k$ modules in $\text{IndBan}_{k}$. Here, the multiplication on $\mathscr{A}$ is determined by mapping isomorphically the tensor product $k \hat{\otimes} k$ of the copies of $k$ indexed by $g$ and $g'$ to the $gg'$ copy of $k$ in $\mathscr{A}$, with the unit being the map from $k$ to the copy of $k$ indexed by $1$. The comultiplication on $\mathscr{A}$ maps the copy of $k$ indexed by $g$ isomorphically to the tensor product $k \hat{\otimes} k$ of the copies of $k$ indexed by $g$ in $\mathscr{A} \hat{\otimes} \mathscr{A}$.
\end{prop}

\begin{proof}
This follows from Theorem \ref{CoIndBanachTannaka}, noting that $FG \cong \mathscr{A} \hat{\otimes} -$.
\end{proof}

\begin{remark}
Since $\mathscr{A}$ is an essentially monomorphic object of $\text{IndBan}_{k}$, we may consider the underlying ring structure of $\mathscr{A}$. This is just $t^{g} \cdot t^{g'} = t^{gg'}$, for $t^{g}$ representing the unit in the copy of $k$ indexed by $g$. The comultiplication on $\mathscr{A}$  is  $t^{g} \mapsto t^{g} \hat{\otimes} t^{g}$, with counit $t^{g} \mapsto \delta_{g,1}$.
\end{remark}

\begin{defn}
Let $\Gamma\text{-}\text{IndBan}_{k}^{\leq 1}$ be the full subcategory of $\Gamma\text{-}\text{IndBan}_{k}$ consisting of IndBanach spaces with an isometric action of $\Gamma$. By this we mean that an object $V$ of $\Gamma\text{-}\text{IndBan}_{k}^{\leq 1}$ can be written as $V=\text{"colim"}_{i \in I}V_{i}$ where the action of $g \in \Gamma$ maps each $V_{i}$ isometrically into some other $V_{i'}$. We will continue to denote the restriction of $F$ to $\Gamma\text{-}\text{IndBan}_{k}^{\leq 1}$ as $F$. $\Gamma\text{-}\text{IndBan}_{k}^{\leq 1}$ is again monoidal, and $F$ is strong monoidal.
\end{defn}

\begin{lem}
With notation as above, asking for an action of $\Gamma$ on an IndBanach space $V$ to be isometric is equivalent to asking that the action of $\Gamma$ on $V$ be bounded. That is, $\{ \|g \cdot :V_{i} \rightarrow V_{i'} \| \mid g \in \Gamma \}$ is bounded.
\end{lem}

\begin{proof}
We can replace the norms on each $V_{i}$ with the equivalent norm $v \mapsto \text{Sup}_{g \in \Gamma} \|gv\|$.
\end{proof}

The following have proofs analogous to those of \ref{DiscreteGroupsAdjunction} and Proposition \ref{RecoveringBornologicalGroupAlgebra}.

\begin{lem}
The forgetful functor $F$ again has a left adjoint, $G':X \mapsto \coprod_{g \in \Gamma}^{\leq 1} X$, where the action of $\Gamma$ on $G'X$ is defined analogously to that on $GX$ in Lemma 3.10.
\end{lem}

\begin{prop}
\label{BanachGroupAlgebra}
$\Gamma\text{-}\text{IndBan}_{k}^{\leq 1}$ is equivalent to $\mathscr{A}'$ modules for the Banach bialgebra $\mathscr{A}' = \coprod_{g \in \Gamma}^{\leq 1}k$, with bialgebra structure defined similarly to $\mathscr{A}$.
\end{prop}

\begin{remark}
$\mathscr{A}'$ is often referred to as the \emph{Banach group algebra}, denoted $l^{1}(\Gamma)$.
\end{remark}

\begin{remark}
Note that the forgetful functors from $\Gamma\text{-}\text{IndBan}_{k}$ and $\Gamma\text{-}\text{IndBan}_{k}^{\leq 1}$ also have right adjoints, $X \mapsto  \prod_{\Gamma}X$ and $X \mapsto \prod_{\Gamma}^{\leq 1}X$, with similar $\Gamma$-actions to $G(X)$ and $G'(X)$. However these functors are not cocontinuous, so our monad is not isomorphic to tensoring with a coalgebra, unless $\Gamma$ is finite. There are still natural morphisms $\coprod_{\Gamma}\prod_{\Gamma} X \rightarrow X$ and $X \rightarrow \prod_{\Gamma} \coprod_{\Gamma} X$, and $\coprod_{\Gamma}^{\leq 1}\prod_{\Gamma}^{\leq 1} X \rightarrow X$ and $X \rightarrow \prod_{\Gamma}^{\leq 1} \coprod_{\Gamma}^{\leq 1} X$, exhibiting an adjunction. If $\Gamma$ is finite then $\mathscr{A}=\mathscr{A}'=l^{1}(\Gamma)$ is dualisable, with dual $l^{\infty}(\Gamma)$.
\end{remark}

\subsection{Representations of topological groups}

\begin{defn}
For a locally compact topological group $H$ and a Banach space $V$ let us denote by $C(H,V)$ the topological vector space of continuous functions $H \rightarrow V$, with the topology of uniform convergence on compact subsets. Let us denote by $C_{b}(H,V)$ the closed subspace of functions which are bounded on $H$, which forms a Banach space with the supremum norm. Note that if $H$ is compact then all continuous functions are bounded, so $C_{b}(H,V)=C(H,V)$. We will use $C_{b}^{\text{lu}}(H,V)$ to denote the closed subspace of \emph{left uniformly continuous} functions. That is, the subspace of functions $f:H \rightarrow V$ such that, for each net $(h_{\lambda})_{\lambda \in \Lambda}$ in $H$ converging to the identity, $\text{Sup}_{x \in H}\|f(h_{\lambda}x)-f(x)\|$ converges to $0$. It was remarked to the author by Anton Lyubinin that if $H$ is compact then all continuous functions must be left uniformly continuous, by a variation of the Heine-Cantor Theorem, in which case $C_{b}^{\text{lu}}(H,V)=C_{b}(H,V)=C(H,V)$.
\end{defn}

Let us fix a locally compact topological group $G$.

\subsubsection{Topological groups with a continuous action by isometries}
\label{TopologicalGroupsExample}

\begin{defn}
\label{GModIso}
Let $G\text{-Mod}^{\text{iso}}$ be the category of strongly continuous IndBanach $G$ modules for which $G$ acts by isometries. That is, the action of $G$ on $V=\text{"colim"}_{i \in I}V_{i}$ is determined by continuous maps $G \rightarrow \text{Hom}(V_{i},V_{i'})$ for each $i \in I$ and for some $i' \in I$ depending on each $i$, where $\text{Hom}(V_{i},V_{i'})$ is given the strong operator topology, whose images lie in the subspace of isometries. The diagonal action of $G$ makes $G\text{-Mod}^{\text{iso}}$ monoidal. We denote by $F$ the forgetful functor to $\text{IndBan}_{k}$.
\end{defn}

\begin{defn}
\label{IndBanachCblu(G,V)}
For a Banach space $V$, let $C_{\text{b}}(G,V)$ be the Banach space of bounded continuous functions from $G$ to $V$, and let $C_{\text{b}}^{\text{lu}}(G,V)$ be the closed subspace of left uniformly continuous functions. For a general IndBanach space $V=\text{"colim"}_{i \in I}V_{i}$ we set $C_{\text{b}}^{\text{lu}}(G,V) = \text{"colim"}_{i \in I}C_{\text{b}}^{\text{lu}}(G,V_{i})$.
\end{defn}

\begin{remark}
In Definition \ref{GModIso}, our representations are in some sense locally Banach. Likewise, our definition of $C_{\text{b}}^{\text{lu}}(G,-)$ in Definition \ref{IndBanachCblu(G,V)} as a functor is in some sense local.
\end{remark}

\begin{lem}[\cite{AFoBC}]
The functor $C_{\text{b}}^{\text{lu}}(G,-)$ is right adjoint to the forgetful functor $F$.
\end{lem}

\begin{proof}
This is proved by Bühler in \cite{AFoBC} for Banach spaces but follows for IndBanach spaces too.
\end{proof}

\begin{prop}
$G\text{-Mod}^{\text{iso}}$ is equivalent to the category of coalgebras over the monoidal comonad $C_{\text{b}}^{\text{lu}}(G,-)$.
\end{prop}

\begin{proof}
This follows from Lemma \ref{CoFibreFunctorsComonadic}.
\end{proof}

\begin{corollary}
In the case where $G$ is compact, $G\text{-Mod}^{\text{iso}}$ is equivalent to the category of comodules over the bialgebra $C_{\text{b}}^{\text{lu}}(G,k)$. Here, the multiplication is pointwise, and the comultiplication is given by the composition
$$C_{\text{b}}^{\text{lu}}(G,k) \overset{\Delta}{\longrightarrow} C_{\text{b}}^{\text{lu}}(G,C_{\text{b}}^{\text{lu}}(G,k))  \cong C_{\text{b}}^{\text{lu}}(G,k) \hat{\otimes} C_{\text{b}}^{\text{lu}}(G,k),$$
with $\Delta(f)(g)(g')=f(gg')$.
\end{corollary}
\begin{proof}
If $G$ is compact, $C_{\text{b}}^{\text{lu}}(G,-)$ is cocontinuous and commutes with contracting colimits, so is isomorphic to $C_{\text{b}}^{\text{lu}}(G,k) \hat{\otimes} -$ by Lemma \ref{TensorFunctorClassification}, and $G\text{-Mod}^{\text{iso}}$ is equivalent to IndBanach $C_{\text{b}}^{\text{lu}}(G,k)$-comodules. Then the monoidal structure gives $C_{\text{b}}^{\text{lu}}(G,k)$ the usual algebra structure arising from pointwise multiplication.
\end{proof}

\subsubsection{Topological Groups with a continuous action, not necessarily by isometries}
\

We now consider a wider class of representations of a topological group. Suppose, for simplicity, that we can write $G$ as a union of compact open subgroups $G=\bigcup_{i \in I}G_{i}$.

\begin{defn}
\label{StronglyContinuousAction}
Let $G\text{-Mod}$ be the category of $k$-IndBanach spaces $V$ with a \emph{strongly continuous} action of $G$. By this we mean an IndBanach space $V$ such that, for each $i \in I$ there is an inductive system of Banach spaces $(V_{j})_{j \in J}$ and map $J \rightarrow J$, $j \mapsto j'$, such that $V \cong \text{"colim"}_{j \in J}V_{j}$ and the action of $G$ on $V$ is induced by continuous maps $G_{i} \rightarrow \text{Hom}(V_{j},V_{j'})$ where $\text{Hom}(V_{j},V_{j'})$ is given the strong operator topology. We will denote by $F$ the forgetful functor from $G\text{-Mod}$ to the category of IndBanach spaces. The diagonal action of $G$ makes $G\text{-Mod}$ monoidal, with trivial action on the monoidal unit $k$, and $F$ is strong monoidal.
\end{defn}

\begin{remark}
If $V \in G\text{-Mod}$ is a Banach space then this just means that the action by $G$ is strongly continuous in the usual sense.
\end{remark}

\begin{remark}
Note that $G\text{-Mod}^{\text{iso}}$ sits as a full subcategory of $G\text{-Mod}$.
\end{remark}

\begin{defn}
For any $i \in I$ and for any Banach space $V$, $C^{\text{lu}}(G_{i},V)$ is a Banach space.  For a general IndBanach space $V=\text{"colim"}_{j \in J}V_{j}$ we can view $C^{\text{lu}}(G_{i},V)$ as the colimit $\text{"colim"}_{j \in J}C^{\text{lu}}(G_{i},V_{j})$ in $\text{IndBan}_{k}$, and we view $C^{\text{lu}}(G,V)$ as the limit $\text{lim}_{i \in I}C^{\text{lu}}(G_{i},V)$. $C^{\text{lu}}(G,V)$ has a left action of $g \in G$ induced by the right regular actions of $G_{i}$ on $C^{\text{lu}}(G_{i},V_{j})$.
\end{defn}

\begin{lem}
\label{DescriptionOfLUFunctions}
$C^{\text{lu}}(G,V)$ can be expressed as the colimit of spaces
$$ \left\lbrace (f_{i})_{i \in I} \in \prod\nolimits_{i \in I}^{\leq 1} C^{\text{lu}}(G_{i},V_{j_{i}})_{r_{i}} \middle| \substack{ \text{For all } i \leq i' \text{ there exists } j \geq j_{i},j_{i'} \\ \text{with } \phi_{j_{i},j} \circ f_{i}|_{G_{i'}} = \phi_{j_{i'},j} \circ f_{i'}} \right\rbrace $$
indexed over pairs $((j_{i})_{i \in I},(r_{i})_{i \in I})$ where $(j_{i})_{i \in I}$ is a collection of indecies in $J$ and $(r_{i})_{i \in I}$ is a collection of positive real numbers, both indexed over $I$. Here, $\phi_{j,j'}:V_{j} \rightarrow V_{j'}$ are the transition maps in the inductive system $(V_{j})_{j \in J}$.
\end{lem}
\begin{proof}
Firstly, note that $C^{\text{lu}}(G,V)$ is the kernel of the map
$$\prod_{i \in I} C^{\text{lu}}(G_{i},V) \rightarrow \prod_{(i \leq i') \in I} C^{\text{lu}}(G_{i'},V)$$
defined by $\pi_{i , i'}=\pi_{i'}-\rho_{i,i'} \circ \pi_{i}$ where $\pi_{i , i'}$ and $\pi_{i}$ are the respective projections and $\rho_{i,i'}:C^{\text{lu}}(G_{i'},V) \rightarrow C^{\text{lu}}(G_{i},V)$ is the restriction map. By the explicit description of limits in \cite{LaACH},
$$\prod_{i \in I} C^{\text{lu}}(G_{i},V) = \text{"colim"}_{(j_{i})_{i \in I},(r_{i})_{i \in I}} \prod_{i \in I}^{\leq 1} C^{\text{lu}}(G_{i},V_{j_{i}})_{r_{i}}$$
and likewise
$$\prod_{(i \leq i') \in I} C^{\text{lu}}(G_{i'},V) = \text{"colim"}_{(j_{i,i'})_{(i \leq i') \in I},(r_{i , i'})_{(i \leq i') \in I}} \prod_{(i \leq i') \in I}^{\leq 1} C^{\text{lu}}(G_{i'},V_{j_{i , i'}})_{r_{i , i'}}.$$
The result then follows by direct computation, again using \emph{loc. cit.}, of this kernel.
\end{proof}

\begin{prop}
The action of $G$ on $C^{\text{lu}}(G,V)$ is strongly continuous for any IndBanach space $V$.
\end{prop}

\begin{proof}
Note that, for any fixed $i_{0} \in I$, we may replace $I$ with $I_{\geq i_{0}}$. In which case, $G_{i_{0}}$ has a strongly continuous action on the spaces describes in Lemma \ref{DescriptionOfLUFunctions}.
\end{proof}

\begin{defn}
For $V$ in $G\text{-Mod}$ with action maps $\pi_{V}^{i,j,j'}:G_{i} \rightarrow \text{Hom}(V_{j},V_{j'})$ we get a collection of bounded linear map $V_{j} \rightarrow C^{\text{lu}}(G_{i},V_{j'})$, $v \mapsto \pi_{V}^{i,j,j'}(-)(v)$, where $V=\text{"colim"}_{i \in I}V_{i}$. These then induce morphisms $V \rightarrow C^{\text{lu}}(G_{i},V)$ in $\text{IndBan}_{k}$, inducing in turn a map $\pi_{V}^{\ast}:V \rightarrow C^{\text{lu}}(G,V)$, the \emph{adjoint} of the action.
\end{defn}

\begin{lem}
\label{StronglyContinuousAdjunction}
The forgetful functor $F:G\text{-Mod} \rightarrow \text{IndBan}_{k}$ has a right adjoint $C^{\text{lu}}(G,-)$.
\end{lem}

\begin{proof}
For an object $V$ of $G\text{-Mod}$, with underlying IndBanach space $FV$, and an IndBanach space $W$, there is a natural map
$$\text{Hom}_{\text{IndBan}_{k}}(FV,W) \rightarrow \text{Hom}_{G}(V, C^{\text{lu}}(G,W)),$$
taking $f$ to the composite $V \overset{\pi_{V}^{\ast}}{\longrightarrow} C^{\text{lu}}(G,V) \overset{f\circ -}{\longrightarrow} C^{\text{lu}}(G,W)$.  Given $i \in I$, the restriction of the map
$$\text{Hom}_{\text{IndBan}}(V, C^{\text{lu}}(G,W)) \rightarrow \text{Hom}_{\text{IndBan}}(V,C^{\text{lu}}(G_{i},W)) \rightarrow \text{Hom}_{\text{IndBan}}(V,W)$$
to $\text{Hom}_{G}(V,C^{\text{lu}}(G,W))$ provides an inverse where the first arrow is induced by the restriction map $C^{\text{lu}}(G,W) \rightarrow C^{\text{lu}}(G_{i},W)$ and the second arrow is induced by the map $C^{\text{lu}}(G_{i},W) \rightarrow W$ that essentailly evaluates a function at $1 \in G_{i} \subset G$ (coming from the maps $C^{\text{lu}}(G_{i},W_{j}) \rightarrow W_{j}$ for $W=\text{"colim"}_{j \in J}W_{j}$). Hence $\text{Hom}_{\text{IndBan}_{k}}(V,W) \cong \text{Hom}_{G}(V, C^{\text{lu}}(G,W))$.
\end{proof}

The following proposition then follows from Lemma \ref{CoFibreFunctorsComonadic}.

\begin{prop}
$G\text{-Mod}$ is equivalent to the category of IndBanach spaces with a coaction of the comonad $C^{\text{lu}}(G,-)$.
\end{prop}

\begin{remark}
Here, the comultiplication $\Delta_{V}:C^{\text{lu}}(G,V) \rightarrow C^{\text{lu}}(G, C^{\text{lu}}(G,V))$ can be thought of as $\Delta(f)(g)(g')=f(gg')$ with counit $f \mapsto f(1)$.
\end{remark}

\begin{corollary}
If $G$ is compact then $G\text{-Mod}$ is equivalent to the monoidal category of IndBanach $C^{\text{lu}}(G,k)$-comodules. Here, the multiplication on $C^{\text{lu}}(G,k)$ is pointwise.
\end{corollary}
\begin{proof}
If $G$ is compact, this monad is isomorphic to $C^{\text{lu}}(G,k) \hat{\otimes}_{k} -$.
\end{proof}

\begin{remark}
The above Corollary is not true if $G$ is not assumed to be compact, and $C^{\text{lu}}(G,k)$ is not \emph{a priori} a coalgebra.
\end{remark}

\subsection{Analytic Galois descent}
\

Let $K \subset L$ be two complete valued fields, let $\text{IndBan}_{K}$ and $\text{IndBan}_{L}$ be their respective categories of IndBanach spaces, let $\text{Hom}_{K}(-,-)$ and $\text{Hom}_{L}(-,-)$ be their morphisms, and let $\hat{\otimes}_{K}$ and $\hat{\otimes}_{L}$ be their monoidal structures. We assume throughout that $L$ is flat over $K$, which is automatic if we are working in the non-Archimedean case by Lemma 3.49 of \cite{SDiBAG}.

\begin{defn}
Let $\text{Res}_{K}^{L}:\text{IndBan}_{L} \rightarrow \text{IndBan}_{K}$ be the restriction functor that restricts $L$-IndBanach spaces to $K$-IndBanach spaces, and let   $\text{Ind}_{K}^{L}:\text{IndBan}_{K} \rightarrow \text{IndBan}_{L}$ be the induction functor $X \mapsto L \hat{\otimes}_{K} X$.
\end{defn}

\begin{lem}
\label{GaloisDescentAdjunction}
$\text{Ind}_{K}^{L}$ and $\text{Res}_{K}^{L}$ form an adjunction, $\text{Hom}_{L}(L \hat{\otimes}_{K} X,Y) \cong \text{Hom}_{K}(X,Y)$, for each $K$-IndBanach space $X$ and $L$-IndBanach space $Y$, thought of as also being a $K$-IndBanach space.
\end{lem}

\begin{proof}
This adjunction is clear when we restrict $X$ and $Y$ to being Banach spaces. Taking colimits then gives the result.
\end{proof}

\begin{remark}
From the above Lemma we obtain a monad $\text{Rest}_{K}^{L}\text{Ind}_{K}^{L} \cong L \hat{\otimes}_{K} -$ on $\text{IndBan}_{K}$, where the resulting $K$-algebra structure on $L$ is the obvious one. It is clear that the restriction functor satisfies the conditions of Barr-Beck, and so, unsurprisingly, $\text{IndBan}_{L}$ is equivalent to the category of $K$-IndBanach spaces with an action of $L$.
\end{remark}

\begin{prop}
$\text{IndBan}_{K}$ is equivalent to objects in $\text{IndBan}_{L}$ with a coaction by $U \cong L \hat{\otimes}_{K} -$ via the functor $X \mapsto L \hat{\otimes}_{K} X$ for $K$-IndBanach spaces $X$.
\end{prop}

\begin{proof}
We obtain a comonad $U=\text{Ind}_{K}^{L}\text{Rest}_{K}^{L} = L \hat{\otimes}_{K}  - $ on $\text{IndBan}_{L}$ from the adjunction in Lemma \ref{GaloisDescentAdjunction}. The comonad structure on $U$ has comultiplication given by the composition $L \hat{\otimes}_{K} Y \cong L \hat{\otimes}_{K} K \hat{\otimes}_{K} Y \rightarrow L \hat{\otimes}_{K} L \hat{\otimes}_{K} Y$, with counit given by scalar multiplication by $L$ on each $L$-IndBanach space $Y$. Since $L$ is assumed to be flat over $K$, the proof of Lemma \ref{CoFibreFunctorsComonadic} then gives our result.
\end{proof}

\begin{remark}
Note that this differs from the general theory outlined previously since $U$ is not $L$-linear, only $K$-linear. Thus we introduce the following framework to deal with this.
\end{remark}

\begin{defn}
For algebras $\mathcal{R}$ and $\mathcal{S}$ in $\text{IndBan}_{K}$, let us denote by $\mathcal{R}\text{-}\mathcal{S}\text{-IndBan}_{K}$ the category of $K$-IndBanach spaces with a left action by $\mathcal{R}$ and right action by $\mathcal{S}$ that are compatible. Then, for $K$-IndBanach algebras $\mathcal{R}$, $\mathcal{S}$, $\mathcal{T}$ and objects $M \in \mathcal{R}\text{-}\mathcal{S}\text{-IndBan}_{K}$ and $N \in \mathcal{S}\text{-}\mathcal{T}\text{-IndBan}_{K}$ we obtain an object $M \hat{\otimes}_{\mathcal{S}} N$ in $\mathcal{R}\text{-}\mathcal{T}\text{-IndBan}_{K}$ as the coequaliser of the two maps $M \hat{\otimes}_{K} \mathcal{S} \hat{\otimes}_{K} N \rightrightarrows M \hat{\otimes}_{K} N$. In particular, this gives $\mathcal{R}\text{-}\mathcal{R}\text{-IndBan}_{K}$ a monoidal structure, $\hat{\otimes}_{\mathcal{R}}$. Suppose now that $\mathcal{R}$ and $\mathcal{S}$ are commutative. For left $\mathcal{R}$ modules (respectively right $\mathcal{S}$ modules) $M$ and $N$ we may view $M \otimes_{K} N$ as a left $\mathcal{R}$ module (resp. right $\mathcal{S}$ module) in two ways depending on whether we act on $M$ or $N$. Thus, for $M, N \in \mathcal{R}\text{-}\mathcal{S}\text{-IndBan}_{K}$, there are four morphisms $\mathcal{R} \hat{\otimes}_{K} (M \hat{\otimes}_{K} N) \hat{\otimes}_{K} \mathcal{S} \rightarrow M \hat{\otimes}_{K} N$. The coequaliser of these four maps, which we denote by $M \hat{\otimes}_{\mathcal{R}\text{-}\mathcal{S}} N$, has a natural left action by $\mathcal{R}$ and right action by $\mathcal{S}$, hence gives an object in $\mathcal{R}\text{-}\mathcal{S}\text{-IndBan}_{K}$. In particular, this gives $\mathcal{R}\text{-}\mathcal{R}\text{-IndBan}_{K}$ a second monoidal structure, which we shall denote by $\hat{\otimes}_{\mathcal{R}\text{-}\mathcal{R}}$.
\end{defn}

\begin{lem}
\label{WeakTensorFunctorClassification}
A functor $\mathscr{V}:\text{IndBan}_{L} \rightarrow \text{IndBan}_{L}$ is isomorphic to one of the form $V \hat{\otimes}_{L} -$ for some $V \in L\text{-}L\text{-IndBan}_{K}$ if and only if cocontinuous functor, enriched over $\text{IndBan}_{K}$, that commutes with $l^{1}$.
\end{lem}

\begin{proof}
This is entirely similar to the proof of Lemma \ref{TensorFunctorClassification}. The main difference is that $\mathscr{V}_{i}(a_{x,y}\iota_{y})$ is not equal to $a_{x,y}\mathscr{V}_{i}(\iota_{y})$ with the usual left $L$ action on $Y_{j_{i}}$. As a result $V=\mathscr{V}(L)$ now has two actions of $L$. On the left, $\lambda \in L$ acts by $\lambda \cdot \text{id}_{\mathscr{V}(L)}$, whilst on the right $\lambda$ acts by $\mathscr{V}(\lambda \cdot \text{id}_{L})$.
\end{proof}

\begin{prop}
\label{GaloisDescent}
$\text{IndBan}_{K}$ is equivalent to the category of left $(L \hat{\otimes}_{K} L)$-comodules in $\text{IndBan}_{L}$ via the induction functor. Here, $(L \hat{\otimes}_{K}L)$ is not a bialgebra in $\text{IndBan}_{L}$ but instead in $L\text{-}L\text{-IndBan}_{K}$ with respect to the monoidal structure $\hat{\otimes}_{L}$. The comultiplication on $(L \hat{\otimes}_{K}L)$ is given by
$$(a \otimes b) \mapsto (a \otimes 1) \otimes (1 \otimes b)$$
and the counit is just multiplication in $L$.
\end{prop}

\begin{remark}
In \cite{CT}, Deligne refers to objects such as $(L \hat{\otimes}_{K} L)$ as \emph{groupoides},  or, in this particular case, \emph{cogebroides}.
\end{remark}

\begin{prop}
With respect to the equivalence in the above Proposition, the monoidal structure of $\text{IndBan}_{K}$ corresponds to the algebra structure on $(L \hat{\otimes}_{K} L)$ given by $(a \otimes b) \cdot (a' \otimes b')=aa' \otimes bb'$, with unit $1 \otimes 1$. Note that this algebra structure is with respect to the tensor product $\hat{\otimes}_{L\text{-}L}$ on $L\text{-}L\text{-IndBan}_{K}$.
\end{prop}

\begin{defn}
Consider $\text{Hom}_{K}(L,L)$ as an object of $L\text{-}L\text{-IndBan}_{K}$ with left action $(\lambda \cdot f)(a)=\lambda f(a)$ and right action $(f \cdot \lambda)(a)=f(\lambda \cdot a)$ for $\lambda,a \in L$, $f \in \text{Hom}_{K}(L,L)$. Then composition gives $\text{Hom}_{K}(L,L)$ an algebra structure with respect to $\hat{\otimes}_{L}$.
\end{defn}

\begin{prop}
We have a non-degenerate pairing
$$\text{Hom}_{K}(L,L) \hat{\otimes}_{L} (L \hat{\otimes}_{K} L) \rightarrow L, \, \, \, \langle f,a \otimes b \rangle = f(a)b,$$
of an algebra with a coalgebra. That is, with the induced pairing between $\text{Hom}_{K}(L,L) \hat{\otimes}_{L} \text{Hom}_{K}(L,L)$ and $(L \hat{\otimes}_{K} L) \hat{\otimes}_{L} (L \hat{\otimes}_{K} L)$ given by
$$\langle f \otimes f', (a \otimes b) \otimes (a' \otimes b')\rangle=\langle f \langle f', a \otimes b \rangle, a' \otimes b' \rangle =\langle f ,\langle f', a \otimes b \rangle a' \otimes b' \rangle,$$
we have that $\langle f \circ g, a \otimes b \rangle =\langle f \otimes g, \Delta(a \otimes b) \rangle$.
\end{prop}

\begin{proof}
$\langle f \circ g, a \otimes b \rangle = f(g(a))b = \langle f, (g(a) \cdot 1)1 \otimes b \rangle = \langle f \otimes g, \Delta(a \otimes b) \rangle$.
\end{proof}

\begin{defn}
Let $\Delta:\text{Hom}_{K}(L,L) \rightarrow \text{Hom}_{K}(L \hat{\otimes}_{K} L,L)$ be the $L$-linear bounded map $\Delta(f)(a \otimes b)=f(ab)$. If $L/K$ is finite then $\text{Hom}_{K}(L \hat{\otimes}_{K} L,L) \cong \text{Hom}_{K}(L,L) \hat{\otimes}_{L\text{-}L} \text{Hom}_{K}(L,L)$ and so $\Delta$ can be viewed as a comultiplication.
\end{defn}

\begin{prop}
We can pair $\text{Hom}_{K}(L \hat{\otimes}_{K} L,L)$ with $(L \hat{\otimes}_{K} L) \hat{\otimes}_{L\text{-}L} (L \hat{\otimes}_{K} L)$, $\langle f, (a \otimes b) \otimes (a' \otimes b') \rangle = f(a \otimes a')bb'$, $f \in \text{Hom}_{K}(L \hat{\otimes}_{K}L, L)$, $a,a',b,b' \in L$. In which case $\langle \Delta(f), (a \otimes b) \otimes (a' \otimes b') \rangle = \langle f, (a \otimes b)\cdot (a' \otimes b') \rangle$.
\end{prop}

\begin{remark}
As a bialgebra, $L \hat{\otimes}_{K} L$ can be thought of as dual to $\text{Hom}_{K}(L,L)$. Since the Galois group, $\Gamma = \Gamma_{L/K}$, sits as the group-like elements within $\text{Hom}_{K}(L,L)$, we may think of $L \hat{\otimes}_{K} L$ as functions on the Galois group. We shall make this more precise. Since $\Gamma$ is a profinite, hence compact, topological group, its strongly continuous $L$-IndBanach representations should fit in the framework of Section \ref{TopologicalGroupsExample}. Since $\Gamma$ does not act $L$-linearly, only $K$-linearly, we must modify the example slightly.
\end{remark}

\begin{defn}
Let $\Gamma\text{-Mod}_{L}$ be the category of $L$-IndBanach spaces $V$ with a strongly continuous action on $\text{Res}_{K}^{L}(V)$ as in Definition \ref{StronglyContinuousAction}, given by $\pi_{V,i,i'}:\Gamma \rightarrow \text{Hom}_{K}(V_{i},V_{i'})$ for $V \cong \text{"colim"}_{i \in I}V_{i}$, such that $\pi_{V,i,i'}(\sigma)(\lambda v)=\sigma(\lambda)\pi_{V,i,i'}(v)$ for $\lambda \in L$, $v \in V_{i}$ and $\sigma \in \Gamma$. Let $F$ be the forgetful functor to $\text{IndBan}_{L}$. The diagonal action of $\Gamma$ makes $\Gamma\text{-Mod}_{L}$ monoidal, with $F$ strong monoidal. Let, for a Banach space $W$, $\tilde{C}^{\text{lu}}(\Gamma,W)$ be the $K$-Banach space of left uniformly continuous functions from $\Gamma$ to $W$ extended to an $L$-Banach space with the twisted action $(\lambda \cdot f)(\sigma)=\sigma(\lambda)f(\sigma)$ for $\lambda \in L$ and $f \in \tilde{C}^{\text{lu}}(\Gamma,W)$. For $W=\text{"colim"}_{i \in I}W_{i}$ an IndBanach space we define $\tilde{C}^{\text{lu}}(\Gamma,W)=\text{"colim"}_{i \in I}\tilde{C}^{\text{lu}}(\Gamma,W_{i})$.
\end{defn}

\begin{lem}
\label{WeakStronglyContinuousAdjunction}
The forgetful functor $F$ has a left adjoint $\tilde{C}^{\text{lu}}(\Gamma,-)$.
\end{lem}
\begin{proof}
The $K$-linear adjoint map $\pi_{V}^{\ast}:V \rightarrow C^{\text{lu}}(\Gamma,V)$ extends to an $L$-linear map $\pi_{V}^{\ast}:V \rightarrow \tilde{C}^{\text{lu}}(\Gamma,V)$. The rest follows as in the proof of Lemma \ref{StronglyContinuousAdjunction}.
\end{proof}

\begin{prop}
The category $\Gamma\text{-Mod}_{L}$ is equivalent to monoidal category of left $C^{\text{lu}}(\Gamma,L)$-comodules in $\text{IndBan}_{L}$. Here, $C^{\text{lu}}(\Gamma,L)$ is an object of $L\text{-}L\text{-IndBan}_{K}$ with left action by $L$ as described for $\tilde{C}^{\text{lu}}(\Gamma,L)$ and right action by $L$ the usual pointwise action on $C^{\text{lu}}(\Gamma,L)$. The multiplication is pointwise, and with respect to $\hat{\otimes}_{L\text{-}L}$, and comultiplication given by the composition
$$C^{\text{lu}}(\Gamma,L) \overset{\Delta}{\longrightarrow} C^{\text{lu}}(\Gamma,C^{\text{lu}}(\Gamma,L)) \cong C^{\text{lu}}(\Gamma,L) \hat{\otimes}_{L} C^{\text{lu}}(\Gamma,L)$$
where $\Delta(f)(\sigma)(\tau)=f(\tau \sigma)$ for $f \in C^{\text{lu}}(\Gamma,L)$, $\sigma, \tau \in \Gamma$.
\end{prop}
\begin{proof}
This follows from Lemma \ref{WeakStronglyContinuousAdjunction}, Lemma \ref{CoFibreFunctorsComonadic} and Lemma \ref{WeakTensorFunctorClassification}.
\end{proof}

\begin{lem}
\label{DescentComparisonFunction}
There is a morphism $\phi: L \hat{\otimes}_{K} L \rightarrow C^{\text{lu}}(\Gamma,L)$, given by
$$\phi(a \otimes b)(\sigma) = \sigma(a)b,$$
that is compatible with the multiplication and comultiplication, and has norm $\|\phi\|= 1$.
\end{lem}
\begin{proof}
Firstly, the fact that $\phi(a \otimes b)$ is left uniformly continuous is straightforward to prove. In fact, if $(x_{\lambda})_{\lambda \in \Lambda}$ is a net converging to $1 \in \Gamma$ then $\text{Sup}_{\sigma \in \Gamma}|\phi(a \otimes b)(x_{\lambda}\sigma)-\phi(a \otimes b)(\sigma)|$ eventually becomes constant at $0$. Secondly,
$$\phi(\lambda \cdot (a \otimes b) \cdot \mu)(\sigma)=\sigma(\lambda)\sigma(a)b \mu = \lambda \cdot (\phi(a \otimes b)(\sigma)) \cdot \mu,$$
$$\phi((a\otimes b)(a' \otimes b'))(\sigma) = \sigma(a)\sigma(a')bb'= (\phi(a\otimes b)\cdot \phi(a' \otimes b'))(\sigma),$$
and
$$\Delta(\phi(a \otimes b))(\sigma)(\tau)=\tau\sigma(a)b=(\sigma(a)\cdot\phi(1 \otimes b))(\tau)=(\phi(a \otimes 1) \otimes \phi(1 \otimes b))(\sigma)(\tau)$$
for $a,b,a',b'\lambda,\mu \in L$ and $\sigma,\tau \in \Gamma$. Also, in the Archimedean case,
$$|\phi(\sum\nolimits_{i} a_{i} \otimes b_{i})(\sigma)| = |\sum\nolimits_{i} \sigma (a_{i})b_{i}| \leq \sum\nolimits_{i}| \sigma (a_{i})||b_{i}|=\sum\nolimits_{i}  |a_{i}||b_{i}|$$
for all $a_{i},b_{i} \in L$ and $\sigma \in \Gamma$, hence
$$\text{Sup}_{\sigma \in \Gamma}|\phi(\alpha)(\sigma)| \leq \text{Inf}\{\sum\nolimits_{i}  |a_{i}||b_{i}| \mid \alpha = \sum\nolimits_{i} a_{i} \otimes b_{i}\}$$
for all $\alpha \in L \hat{\otimes}_{K}L$. That is, $\|\phi\| \leq 1$. The non-Archimedean case is similar. The fact that $\|\phi\|=1$ follows since $\phi$ preserves the unit, which is of norm 1 in both spaces.
\end{proof}

\begin{lem}
\label{ContractingColimitOfFields}
Let $L/K$ be an extension of complete valued fields such that the algebraic elements are dense in $L$. Then $L \cong \text{colim}^{\leq 1}_{K \subset L' \subset L}L'$, where this is the contracting colimit taken in $\text{Ban}_{K}$ over all finite extensions $K \subset L'$ contained in $L$.
\end{lem}

\begin{proof}
We have strict monomorphisms $L' \hookrightarrow L$ for all finite extensions $K \subset L'$ contained in $L$. Suppose we are given a compatible collection of bounded linear maps $\{f_{L'}:L' \rightarrow V\}_{K \subset L' \subset L}$ such that $\{\|f_{L'}\|\}_{K \subset L' \subset L}$ is bounded by some $M>0$. Then we obtain a well defined bounded linear map $f:\bigcup_{K \subset L' \subset L} L' \rightarrow V$ defined on each $L'$ by $f_{L'}$. The compatibility of the collection $\{f_{L'}\}_{K \subset L' \subset L}$ ensures that this is well defined. By assumption, $\bigcup_{K \subset L' \subset L} L'$ is dense in $L$, hence we may extend $f$ to a unique map $L \rightarrow V$ such that $f_{L'}$ is the composition $L' \hookrightarrow L \rightarrow V$. Clearly $\|f\| \leq M$.
\end{proof}

\begin{lem}
\label{ContractingColimitOfTensors}
For an extension of complete valued fields, $L/K$, such that the algebraic elements are dense in $L$, there is an isomorphism $L \hat{\otimes}_{K} L \cong \text{colim}^{\leq 1}_{K \subset L' \subset L}L' \hat{\otimes}_{K} L'$.
\end{lem}
\begin{proof}
This follows from Lemma \ref{ContractingColimitOfFields}.
\end{proof}

\begin{lem}
\label{LocallyConstantsDense}
For $G$ a profinite group and $V$ a Banach space, the subspace of locally constant functions is dense in $C^{\text{lu}}(G,V)$.
\end{lem}
\begin{proof}
Let $f:G \rightarrow V$ be a left uniformly continuous function. For a fixed $g_{0} \in \Gamma$, suppose for a contradiction that the net
$$\left(\text{Sup}_{g \in g_{0}N} \|f(g)-f(g_{0})\|\right)_{\substack{N \trianglelefteq G \\ [G;N]<\infty}}$$
does not converge to $0$. Hence there is a sequence $(g_{N})_{N \trianglelefteq G}$ converging to $g_{0}$ such that $\|f(g_{N})-f(g_{0})\|$ does not converge to $0$, which contradicts left uniform continuity of $f$. Thus for all $\varepsilon >0$ there exists $N_{g_{0}} \trianglelefteq G$ such that $\text{Sup}_{g \in g_{0}N_{g_{0}}}\|f(g)-f(g_{0}\|)<\varepsilon$. This means that, by looking at $\{N_{g_{0}} \mid g_{0} \in G\}$ and $f(g_{0}) \in V$, for each $\varepsilon >0$ there exists a cover $\mathcal{U}_{\varepsilon}$ of compact open subsets which has the property that each $U \in \mathcal{U}$ has some $\lambda_{U} \in V$ for which $\text{Sup}_{g \in U} \|f(g)-\lambda_{U}\| <\varepsilon$. By compactness of $G$ we may assume that $\mathcal{U}_{\varepsilon}$ is finite, and furthermore we can take the sets in $\mathcal{U}_{\varepsilon}$ to be pairwise disjoint. We then have that the locally constant function $\sum_{U \in \mathcal{U}}\lambda_{U}\chi_{U}$ approximates $f$, $\|f-\sum_{U \in \mathcal{U}}\lambda_{U}\chi_{U}\| \leq \varepsilon$, in $C^{\text{lu}}(G,V)$.
\end{proof}

\begin{lem}
\label{WeakContractingColimitOfFunctions}
Let $L/K$ be an extension of complete valued fields such that the algebraic elements are dense in $L$ and form a Galois extension over $K$. Then there is an isomorphism $\text{colim}^{\leq 1}_{H \trianglelefteq \Gamma} C^{\text{lu}}(\Gamma/H,L) \overset{\sim}{\longrightarrow} C^{\text{lu}}(\Gamma,L)$, where this is the contracting colimit taken in $\text{Ban}_{K}$ over all finite index normal subgroups $H \trianglelefteq \Gamma$.
\end{lem}

\begin{proof}
A proof similar to that of Lemma \ref{ContractingColimitOfFields} shows that the Banach space $\text{colim}^{\leq 1}_{H \trianglelefteq \Gamma} C^{\text{lu}}(\Gamma/H,L)$ is isomorphic to the closure of $\bigcup_{H \trianglelefteq \Gamma} C^{\text{lu}}(\Gamma,L)^{H}$, since the image of $C^{\text{lu}}(\Gamma/H,L)$ in $C^{\text{lu}}(\Gamma,L)$ is just the $H$ invariant subspace. It follows from the definition of the profinite topology on $\Gamma$ that a function is locally constant if and only if it lies in one of these invariant subspaces. By Lemma \ref{LocallyConstantsDense} this subspace is dense.
\end{proof}

\begin{lem}
\label{ContractingColimitOfFunctions}
For an extension of complete valued fields, $L/K$, such that the algebraic elements are dense in $L$ and form a Galois extension over $K$, there is an isomorphism $\text{colim}^{\leq 1}_{H \trianglelefteq \Gamma} C^{\text{lu}}(\Gamma/H,L^{H}) \overset{\sim}{\longrightarrow} C^{\text{lu}}(\Gamma,L)$, where the contracting colimit is taken in $\text{Ban}_{K}$ over all finite index normal subgroups $H \trianglelefteq \Gamma$.
\end{lem}

\begin{proof}
This follows from Lemma \ref{ContractingColimitOfFields}, Lemma \ref{WeakContractingColimitOfFunctions}, the fact that $C^{\text{lu}}(G,-)$ commutes with contracting colimits for finite discrete groups $G$, and the fact that all finite Galois extensions over $K$ in $L$ are of the form $L^{H}$ for $H \trianglelefteq \Gamma$ of finite index.
\end{proof}

\begin{lem}
\label{DescentComparisonInFiniteDimensions}
If $L/K$ is a finite Galois extension then the morphism $\phi$ in Lemma \ref{DescentComparisonFunction} is an isomorphism.
\end{lem}

\begin{proof}
By the open mapping theorem and Lemma \ref{DescentComparisonFunction}, it is enough to show that $\phi$ is a bijection. First, by the Normal Basis Theorem, we may take be a normal basis $B$ of $L$ over $K$. That is, $B$ is a basis of $L$ over $K$ comprised of a single orbit of the Galois group $\Gamma$. Taking a basis $\{b \otimes 1 \mid b \in B\}$ of $L \hat{\otimes}_{K}L$ over $L$ (with its right action) and the basis $\{\sigma \mapsto \delta_{\sigma,\tau} \mid \tau \in \Gamma\}$ of $C^{\text{lu}}(\Gamma,L)$ over $L$ (with its right action) we see that $\phi$ is given by the matrix with entries $(\tau(b))_{(b,\tau) \in B \times \Gamma}$ indexed over $B \times \Gamma$. The columns of this matrix are all linearly independent since $\Gamma$ permutes $B$ simply transitively, hence it is invertible and so is $\phi$.
\end{proof}

\begin{remark}
It is not clear whether $\phi$ is an isometry in the above finite dimensional case. This means that the norm of $\phi^{-1}$ might become arbitrarily large as we range over an infinite collection of such extensions. Hence $\phi$ may not remain an isomorphism after taking contracting colimits over infinitely many of these finite extensions (using Lemmas \ref{ContractingColimitOfTensors} and \ref{ContractingColimitOfFunctions}). We do, however, have the following result.
\end{remark}

\begin{prop}
\label{DescentComparisonBijectiveOnDenseSubspace}
Let $L/K$ be an extension of complete valued fields such that the algebraic elements, $L^{a}$, are dense in $L$ and form a Galois extension over $K$. Then $\phi$ restricts to a continuous bijection between the dense subspaces $L^{a} \otimes_{K} L \subset L \hat{\otimes}_{K}L$ (the algebraic tensor product of $L^{a}$ with $L$) and the subspace of locally constant functions in $C^{\text{lu}}(\Gamma,L)$.
\end{prop}

\begin{proof}
By Lemma \ref{ContractingColimitOfFields}, there is an isomorphism
$$\text{colim}_{K \subset L' \subset L}^{\leq 1} L' \hat{\otimes}_{K} L \cong L \hat{\otimes}_{K}L$$
in $\text{IndBan}_{L}$ under which the algebraic tensor product $L^{a} \otimes_{K}L$ is the union of the images of $L' \hat{\otimes}_{K} L = L' \otimes_{K} L$. By Lemma \ref{WeakContractingColimitOfFunctions} there is an isomorphism
$$\text{colim}_{K \subset L' \subset L}^{\leq 1} C^{\text{lu}}(\Gamma_{L'/K},L)=\text{colim}_{H \trianglelefteq \Gamma}^{\leq 1} C^{\text{lu}}(\Gamma/H,L) \cong C^{\text{lu}}(\Gamma,L)$$
under which the union of the images of $C^{\text{lu}}(\Gamma_{L'/K},L)$ is the subspace of locally constant functions. The result then follows since $\phi$ restricts to the extension of the continuous bijection in Lemma \ref{DescentComparisonInFiniteDimensions} from each $L' \otimes_{K} L$ to the corresponding $C^{\text{lu}}(\Gamma_{L'/K},L)$.
\end{proof}

\begin{remark}
The above proposition says precisely that $L \hat{\otimes}_{K}L$ is a completion of the space of locally constant functions with respect to a stronger topology than that inherited from $C^{\text{lu}}(\Gamma,L)$. It is in this way that we may think of $L \hat{\otimes}_{K} L$ as functions on the Galois group $\Gamma$.
\end{remark}

\begin{defn}
Let $L/K$ be an extension of complete valued fields such that the algebraic elements, $L^{a}$, are dense in $L$ and form a Galois extension over $K$ with Galois group $\Gamma$.  We think of $L^{a}$ as a formal colimit over finite extensions of $K$ in $L$ in $\text{IndBan}_{K}$, hence as a $K$-IndBanach algebra. We define the IndBanach (or Bornological, following the equivalence in \cite{DGaBAG}) space of locally constant $L$-valued functions on $G$, $C^{\text{lc}}(\Gamma,L)$, to be the colimit
$$C^{\text{lc}}(\Gamma,L)=\text{"colim"}_{N \trianglelefteq \Gamma}C^{\text{lu}}(\Gamma/N,L)$$
taken over finite index normal subgroups of $\Gamma$. Similarly we define the IndBanach (or Bornological) algebraic tensor product, $L^{a} \otimes L$, to be the colimit
$$L^{a} \otimes_{K} L=\text{"colim"}_{K \subset L' \subset L}L' \hat{\otimes}_{K} L$$
taken over finite extensions $L'$ of $K$ in $L$. We may also define
$$C^{\text{lc}}(\Gamma,L^{a})=\text{"colim"}_{N \trianglelefteq \Gamma}C^{\text{lu}}(\Gamma/N,L^{N})=\text{"colim"}_{\substack{N \trianglelefteq \Gamma \\ K \subset L' \subset L}}C^{\text{lu}}(\Gamma/N,L')$$
and
$$L^{a} \otimes_{K} L^{a}=\text{"colim"}_{K \subset L' \subset L}L' \hat{\otimes}_{K} L'=\text{"colim"}_{\substack{K \subset L' \subset L \\ K \subset L'' \subset L}}L' \hat{\otimes}_{K} L''$$
in a similar way.
\end{defn}

We may then rephrase Proposition \ref{DescentComparisonBijectiveOnDenseSubspace} as the following.

\begin{prop}
There is a commutative diagram
\begin{center}
\begin{tikzpicture}[node distance=6cm, auto]
  \node (A) {$C^{\text{la}}(\Gamma,L)$};
  \node (B) [below=0.3cm of A] {$C^{\text{lc}}(\Gamma,L)$};
  \node (C) [right=1.3cm of A] {$L \hat{\otimes}_{K} L$};
  \node (D) [below=0.4cm of C] {$L^{a} \otimes_{K} L$};
  \node (E) [below=0.3cm of B] {$C^{\text{lc}}(\Gamma,L^{a})$};
  \node (F) [below=0.4cm of D] {$L^{a} \otimes_{K} L^{a}$};
  \draw[->] (B) to node [swap]{} (A);
  \draw[->] (A) to node {$\phi$} (C);
  \draw[->] (B) to node {$\sim$} (D);
  \draw[->] (D) to node {} (C);
  \draw[->] (E) to node [swap]{} (B);
  \draw[->] (E) to node {$\sim$} (F);
  \draw[->] (F) to node {} (D);
\end{tikzpicture}
\end{center}
whose vertical arrows are bimorphisms.
\end{prop}

\begin{defn}
Let $\text{Ind}_{\phi}$ be the induction functor
$$\text{Ind}_{\phi}:(L \hat{\otimes}_{K}L)\text{-Comod} \rightarrow C^{\text{lu}}(G,k)\text{-Comod} \cong \Gamma\text{-Mod}_{L}, \, \, \, M \mapsto \text{Ind}_{\phi}M,$$
from the category of $L \hat{\otimes}_{K}L$ comodules in $\text{IndBan}_{L}$ to $\Gamma\text{-Mod}_{L}$, where $\text{Ind}_{\phi}M$ has the same underlying IndBanach space as $M$ but with the coaction
$$M \rightarrow (L \hat{\otimes}_{K}L) \hat{\otimes}_{L} M \overset{\phi \otimes \text{id}_{M}}{\longrightarrow} C^{\text{lu}}(G,k) \hat{\otimes}_{L}M.$$
\end{defn}

\begin{lem}
\label{ExactFaithfulInduction}
The induction functor $\text{Ind}_{\phi}$ is exact and faithful. If we are working in the non-Archimedean case, $\text{Ind}_{\phi}$ is also full.
\end{lem}
\begin{proof}
Exactness and faithfulness follows from the fact that the forgetful functors from these categories are faithful and reflect exactness, and that composition of $\text{Ind}_{\phi}$ with the forgetful functor from $\Gamma\text{-Mod}_{L}$ gives the forgetful functor from $(L \hat{\otimes}_{K}L)\text{-Comod}$. If $f:\text{Ind}_{\phi}M \rightarrow \text{Ind}_{\phi}N$ is a morphism of $C^{\text{lu}}(G,k)$ comodules, where $M$ and $N$ are $(L \hat{\otimes}_{K}L)$ comodules with respective coactions $\Delta_{M}$ and $\Delta_{N}$ then
$$(\phi \otimes \text{Id})\circ \Delta_{N} \circ f = (f \otimes \text{Id})\circ (\phi \otimes \text{Id})\circ \Delta_{M}=(\phi \otimes \text{Id})\circ (f \otimes \text{Id})\circ \Delta_{M}.$$
By Lemma 3.49 of \cite{SDiBAG}, assuming we are working in the non-Archimedean case, $C^{\text{lu}}(G,k)$ is a flat IndBanach space and so $\phi \otimes \text{Id}$ is monic. Hence $f:M \rightarrow N$ is a morphism of $(L \hat{\otimes}_{K}L)$ comodules.
\end{proof}

\begin{defn}
Let $\Gamma\text{-Mod}_{L}^{\text{sm}}$ denote the essential image of $\text{Ind}_{\phi}$ in $\Gamma\text{-Mod}_{L}$, the category of \emph{smooth} representations of $\Gamma$.
\end{defn}

\begin{prop}
The category of smooth representations, $\Gamma\text{-Mod}_{L}^{\text{sm}}$, is equivalent to $\text{IndBan}_{K}$ as monoidal categories via the induction functor
$$\text{IndBan}_{K} \rightarrow \Gamma\text{-Mod}_{L}^{\text{sm}}, \, \, \, V \mapsto L \hat{\otimes}_{K} V.$$
\end{prop}

\begin{proof}
This follows from Proposition \ref{GaloisDescent} and Lemma \ref{ExactFaithfulInduction}.
\end{proof}

\begin{defn}
Let $G$ be a profinite group, $k$ be a complete valued field and $A$ be an IndBanach algebra over $k$. We define the IndBanach (or Bornological) Iwasawa algebra, $\Lambda_{A}^{\text{Born}}(G)$, to be the limit
$$\Lambda_{A}^{\text{Born}}(G)=\text{lim}_{N \trianglelefteq G}A[G/N]$$
in $\text{IndBan}_{k}$ taken over all open normal subgroups of $G$, where $A[G/N]$ is the Banach group algebra $\coprod^{\leq 1}_{G/N}A$ over $A$ defined similarly to the algebra in Proposition \ref{BanachGroupAlgebra}. If $A$ is a Banach algebra then we define the Banach Iwasawa algebra, $\Lambda_{k}^{\text{Ban}}(G)$, as the contracting limit
$$\Lambda_{A}^{\text{Ban}}(G)=\text{lim}^{\leq 1}_{N \trianglelefteq G}A[G/N]$$
in $\text{Ban}_{k}$.
\end{defn}

\begin{prop}
Let $L/K$ be an extension of complete valued fields such that the algebraic elements, $L^{a}$, are dense in $L$ and form a Galois extension over $K$ with Galois group $\Gamma$. Then, as IndBanach spaces over $L$, $\Lambda_{L}^{\text{Ban}}(\Gamma)$ is dual to $C^{\text{lu}}(\Gamma,L)$, and, as $L^{a}$ modules in $\text{IndBan}_{K}$, $\Lambda_{L^{a}}^{\text{Born}}(\Gamma)$ is dual to $C^{\text{lc}}(\Gamma,L^{a}) \cong L^{a} \otimes_{K} L^{a}$.
\end{prop}

\begin{proof}
The first statement follows from the isomorphisms
$$
\begin{array}{rcl}
\text{Hom}_{L}(C^{\text{lu}}(\Gamma,L),L) &=& \text{Hom}_{L}(\text{colim}^{\leq 1}_{N \trianglelefteq \Gamma}C^{\text{lu}}(\Gamma/N,L),L)\\
&\cong& \text{lim}^{\leq 1}_{N \trianglelefteq \Gamma}\text{Hom}_{L}(C^{\text{lu}}(\Gamma/N,L),L)\\
&\cong& \text{lim}^{\leq 1}_{N \trianglelefteq \Gamma}\text{Hom}_{L}(\coprod\nolimits_{\Gamma/N}^{\leq 1}L,L)\\
&\cong& \text{lim}^{\leq 1}_{N \trianglelefteq \Gamma}\coprod\nolimits_{\Gamma/N}^{\leq 1}\text{Hom}_{L}(L,L)\\
&\cong& \text{lim}^{\leq 1}_{N \trianglelefteq \Gamma}\coprod\nolimits_{\Gamma/N}^{\leq 1}L = \Lambda_{L}^{\text{Ban}}(\Gamma).
\end{array}$$
The second follows from
$$
\begin{array}{rcl}
\text{Hom}_{L^{a}}(C^{\text{lu}}(\Gamma,L^{a}),L^{a}) &=& \text{Hom}_{L^{a}}(L^{a} \hat{\otimes} C^{\text{lu}}(\Gamma/N,K),L^{a})\\
&\cong& \text{Hom}_{K}(C^{\text{lu}}(\Gamma,K),L^{a})
\end{array}$$
and a similar argument to the above.
\end{proof}

\begin{remark}
The above isomorphisms are not isomorphisms of algebras. The multiplications on $\Lambda_{L}^{\text{Ban}}(\Gamma)$ and $\Lambda_{L^{a}}^{\text{Born}}(\Gamma)$ induced by the respective comultiplications on $C^{\text{lu}}(\Gamma,L)$ and $C^{\text{lc}}(\Gamma,L^{a})$ are twisted by the actions of $\Gamma$ on $L$ and $L^{a}$. Since there is a faithful embedding of $\Gamma\text{-Mod}_{L}$, viewed as $C^{\text{lu}}(\Gamma,L)$-comodules, into modules over the twisted Iwasawa algebra $\Lambda_{L}^{\text{Ban}}(\Gamma)$ we may alternatively take this action as our descent data to recover a $K$-IndBanach space $V$ from the induced $L$-IndBanach space $L\hat{\otimes}_{K}V$.
\end{remark}

\end{document}